\documentclass[11pt]{amsart}
\usepackage{amsmath, amssymb, amsbsy, amsfonts, amsthm, latexsym, amsopn, amstext, amsxtra, euscript, amscd, color, mathrsfs}
\usepackage[normalem]{ulem}
\usepackage{soul}

\PassOptionsToPackage{hyphens}{url}\usepackage{hyperref}
\makeatletter
\@namedef{subjclassname@2020}{%
  \textup{2020} Mathematics Subject Classification}
\makeatother
 
 \usepackage[capbesideposition=outside,capbesidesep=quad]{floatrow}

\restylefloat{table}
\restylefloat{table}
            
\usepackage{multirow,caption}
\usepackage{adjustbox}

\usepackage{lipsum}   
\usepackage{amscd}
\usepackage{color,enumerate}
\newcommand{\RNum}[1]{\lowercase\expandafter{\romannumeral #1\relax}}

\usepackage[colorinlistoftodos,prependcaption,textsize=tiny]{todonotes}

\newtheorem{thm}{Theorem}[section]
\newtheorem{lem}[thm]{Lemma}

\newtheorem{exmp}[thm]{Example}

\newtheorem{thm-con}[thm]{Theorem-Conjecture}
\numberwithin{equation}{section}

\theoremstyle{definition}

\newcommand{\F}{\mathbb F}

\def\Tr{{\rm Tr}}
\def\Trn{{\rm Tr}_1^n}
\def\Trm{{\rm Tr}_1^m}
\def\Trtm{{\rm Tr}_1^{2m}}
\def\Trmn{{\rm Tr}_m^{3m}}
\def\Trnn{{\rm Tr}_m^{2m}}
\def\Trmm{{\rm Tr}_m^{n}}
\def\Trfn{{\rm Tr}_4^{n}}
\def\Trtn{{\rm Tr}_2^{n}}

\begin{document}
\title[Differential properties of certain permutation polynomials]{The differential properties of certain permutation polynomials over finite fields}

\author[K. Garg]{Kirpa Garg}
\address{Department of Mathematics, Indian Institute of Technology Jammu, Jammu 181221, India}
\email{kirpa.garg@gmail.com}

\author[S. U. Hasan]{Sartaj Ul Hasan}
\address{Department of Mathematics, Indian Institute of Technology Jammu, Jammu 181221, India}
\email{sartaj.hasan@iitjammu.ac.in}

 \author[P.~St\u anic\u a]{Pantelimon~St\u anic\u a}
 \address{Applied Mathematics Department, Naval Postgraduate School, Monterey, CA 93943, USA}
\email{pstanica@nps.edu}

\thanks{The work of K. Garg is supported by the University Grants Commission (UGC), Government of India. 
The work of S. U. Hasan is partially supported by Core Research Grant CRG/2022/005418 from the Science and Engineering Research Board, Government of India. The work of P. St\u anic\u a is partially supported by a grant from the NPS Foundation.}

 \begin{abstract}  Finding functions, particularly permutations, with good differential properties has received a lot of attention due to their varied applications. For instance, in combinatorial design theory, a correspondence of perfect $c$-nonlinear functions and difference sets in some quasigroups was recently shown by Anbar et al. (J. Comb. Des. 31(12):1-24, 2023). Additionally, in a recent manuscript by Pal et al. (Adv.  Math. Communications, to appear), a very interesting connection between the $c$-differential uniformity and boomerang uniformity, when $c=-1$, was pointed out, showing that they are the same for an odd APN permutation, sparking yet more interest in  the construction of functions with low $c$-differential uniformity. We investigate the $c$-differential uniformity of some classes of permutation polynomials. As a result, we add four more classes of permutation polynomials to the family of functions that only contains a few (non-trivial) perfect $c$-nonlinear functions over finite fields of even characteristic. Moreover, we include a class of permutation polynomials with low $c$-differential uniformity over the field of characteristic~$3$. To solve the involved equations over finite fields, we use various number theoretical techniques, in particular, we find explicitly many Walsh transform coefficients and Weil sums that may be of an independent interest.
 \end{abstract}



\maketitle
\section{Introduction} 
Let $\F_{q}$ be the finite field with $q=p^n$ elements, where $p$ is a prime number and $n$ is a positive integer. We use $\F_q[X]$ to denote the ring of polynomials in one variable X with coefficients in $\F_q$ and $\F_{q}^*$ to denote the multiplicative group of nonzero elements of $\F_{q}$. If $F$ is a function from $\F_{q}$ to itself, by using Lagrange's interpolation formula, one can uniquely express it as a polynomial in $\F_{q}[X]$ of degree at most $q-1$. A polynomial $F \in \F_{q}[X]$ is a permutation polynomial of $\F_{q}$ if the mapping $X \mapsto F(X)$ is a bijection on~$\F_{q}$. Permutation polynomials over finite fields are of great interest due to their numerous applications in coding theory~\cite{Ding_C_13, Chapuy_C_07}, combinatorial design theory~\cite{Ding_Co_06}, cryptography~\cite{Lidl_Cr_84, Schwenk_Cr_98}, and other areas of mathematics and engineering. These polynomials with some desired properties such as low differential uniformity, high algebraic degree and high nonlinearity act as important candidates in designing cryptographically strong S-boxes and hence in providing secure communication.

 Borisov et al.~\cite{Borisov} introduced the concept of the multiplicative differentials of the form $(F(cX), F(X))$ and exploited this new class of differentials to attack certain existing ciphers. In 2020, Ellingsen et al.~\cite{Ellingsen} gave the notion of $c$-differential uniformity. For any function $F : \F_{q}\to \F_{q}$ and for any $a , c \in \F_{q}$, the (multiplicative) $c$-derivative of $F$ with respect to $a$ is defined as $_c\Delta_F (X, a) := F(X + a)- cF(X)$ for all $X \in \F_{q}$. For any $a,b \in \F_q$, the $c$-Difference Distribution Table ($c$-DDT) entry $_c\Delta_F(a,b)$ at point $(a,b)$ is the number of solutions { $X \in \F_q$} of the equation $_c\Delta_F(X,a) = b$. The $c$-differential uniformity of $F$, denoted by $_c\Delta_F$, is given by $_c\Delta_{F} := \max\{_c\Delta_{F}(a, b) : a, b \in \F_{q} \hspace{0.2cm}\text{and}\hspace{0.2cm} a \neq 0 \hspace{0.2cm}\text{if} \hspace{0.2cm}c=1\}.$ When $_c\Delta_{F}=1$, we call $F$ a perfect $c$-nonlinear (P$c$N) function and when $_c\Delta_{F}=2$, we call $F$ an almost perfect $c$-nonlinear (AP$c$N) function. Note that for monomial functions, $X\mapsto X^d$, the output differential  $(c_1F(X), F(X))$ is the same as the input differential  $(F(c_2X), F(X))$, where $c_1=c_2^d$, which was the differential  that Borisov et al.~\cite{Borisov}  exploited.
 
 Recently, the authors in~\cite{AMS22} pointed out a connection between the $c$-differential uniformity (cDU) and combinatorial designs by showing that  the graph of a P$c$N function is a difference set in a quasigroup. Difference sets give rise to symmetric designs, which are known to construct optimal self complementary codes. Some types of designs also have applications in secret sharing and visual cryptography. Moreover, recently, Pal and St\u anic\u a~\cite{PS23} showed that the $c$-differential uniformity, when $c=-1$, of an odd APN permutation  $F$ (odd characteristic) equals its boomerang uniformity, and when $F$ is a non-permutation, the boomerang uniformity of $F$ is the maximum of the $(-1)$-DDT entries (without  the first row/column). In view of the applications stated above, the construction of functions, particularly permutations, with low $c$-differential uniformity is an interesting problem, and recent work has focused heavily in this direction. One can refer to~\cite{HPRS20,JKK2,Ozbudak,CPC,MRS,wang,lwz,Hu} for the numerous functions with low $c$-differential uniformity investigated till now. There are very few known (non-trivial, that is, nonlinear) classes of P$c$N and AP$c$N functions   over a finite field with even characteristic (see, for example,~\cite{Garg, HPS1,JKK, TL23}).

In this paper, we compute the $c$-differential uniformity of some classes of permutation polynomials introduced in~\cite{LWC,WBZ,ZH}. The methods we employ can be of independent interest, and they use Walsh transforms computations, Weil sums, and some detailed investigation of the involved equations via number theory tools. The structure of the paper is as follows. We first recall some results in Section~\ref{S2}, that are required in the subsequent sections. In Section~\ref{S3} we compute the $c$-differential uniformity of four classes of permutation polynomials over finite fields of even characteristic. Further, Section~\ref{S4} deals with the $c$-differential uniformity of one class of permutation polynomials over finite fields of characteristic three.
  Finally, in Section~\ref{sec:concl}, we conclude the paper.

\section{Preliminaries}
 \label{S2}
 In this section, we first review a definition and provide some lemmas to be used later in subsequent sections. Throughout the paper, we shall use $\Trmm$ to denote the (relative) trace function from $\F_{p^n} \rightarrow \F_{p^m}$, i.e., $\Trmm(X) = \sum_{i=0}^{\frac{n-m}{m}} X^{p^{mi}}$, where $m$ and $n$ are positive integers and $m|n$.  For $m=1$, we use $\Tr$ to denote the absolute trace. Also $v_p(n)$ will denote the highest nonnegative exponent $v$ such that $p^v$ divides $n$ (that is, the $p$-adic valuation). 
 
We first recall the definition of the Walsh transform of a function (see, for example~\cite{TH}).
 For a function $F : \F_{p^n} \rightarrow \F_{p}$, the Walsh transform of $F$ at $v \in  \F_{p^n}$, is defined as
\[
 \mathcal{W}_F(v) = \sum_{X \in \F_{p^n}} \omega^{F(X) - \Tr(v X)},
 \]
where $\omega = e^{\frac{2 \pi i}{p}}$ is a complex primitive $p$th root of unity.

We now present a lemma dealing with solutions of a cubic equation over a finite field with even characteristic. We prefer the statement from~\cite[Lemma 2.2]{ZH}, which is derived from the original paper of Williams~\cite{W75}.
\begin{lem}\label{L01}\textup{\cite[Lemma 2.2]{ZH}} For a positive integer $n$ and $a \in \F_{2^n}^{*}$, the cubic equation $X^3+X+a=0$ has 
\begin{enumerate}
 \item[$(1)$] a unique solution in $\F_{2^n}$ if and only if $\Trn(a^{-1}+1)=1$;
 \item[$(2)$] three distinct solutions in $\F_{2^n}$ if and only if $p_n(a)=0$, where the polynomial $p_n(X)$ is recursively defined by the equations $p_1(X)=p_2(X)=X, p_k(X)=p_{k-1}(X)+X^{2^{k-3}}p_{k-2}(X)$ for $k \geq 3$;
 \item[$(3)$] no solutions in $\F_{2^n}$, otherwise.
\end{enumerate}
\end{lem}

Next, we recall five classes of permutation polynomials for which we investigate their $c$-differential uniformity. For an element $\delta$ in a finite field $\F_{p^{2m}}$, we let $\bar \delta=\delta^{p^m}$. 

\begin{lem} \textup{\cite[Theorem 3.1]{ZH}}
\label{L02}
Let $m$ be a positive integer and $p_m(X)$ be defined in Lemma~\textup{\ref{L01}}. Let
$\delta\in\F_{2^{2m}}$ satisfy either $\delta\in\F_{2^m}$, or $\delta \not \in  \F_{2^m}$ with $\Trtm(\delta) = \Trm(1)$ and
$p_m((\delta + \bar \delta)^{-1}) \neq 0$. The polynomial
$$
F(X)=(X^{2^m}+X+\delta)^{2^{2m-2}+2^{m-2}+1}+X
$$
permutes $\F_{2^{2m}}$.
\end{lem}

\begin{lem}  \textup{\cite[Theorem 2.1]{ZH}} 
\label{L05}
For a positive integer $m \not \equiv
0 \pmod 3$ and any element $\delta \in \F_{2^{2m}}$ , the
polynomial
$$
F(X)=(X^{2^m}+X+\delta)^{3 \cdot 2^{2m-2}+2^{m-2}}+X
$$
permutes $\F_{2^{2m}}$.
\end{lem}

\begin{lem} \textup{\cite[Theorem 3.1]{WBZ}} 
 \label{L06}
 For a positive integer $m \not \equiv
0 \pmod 3$ and any element $\delta \in \F_{2^{2m}}$ , the
polynomial
$$
F(X)=(X^{2^m}+X+\delta)^{3 \cdot 2^{m-2}+2^{2m-2}}+X
$$
permutes $\F_{2^{2m}}$.
\end{lem}

\begin{lem}  \textup{\cite[Proposition 8]{LWC}}
\label{L03}
For a positive integer $m$ and a fixed $\delta \in \F_{2^{3m}}$ with $\Trmm (\delta) = 0$, where $n=3m$, the polynomial
$$
F(X)=(X^{2^m}+X+\delta)^{{2^{2m+1}}+2^m}+(X^{2^m}+X+\delta)^{2^{2m}+2^{m+1}}+X
$$
permutes $\F_{2^{3m}}$.
\end{lem}

\begin{lem} \textup{\cite[Proposition 10]{LWC}}
\label{L04} 
For a positive integer $m$ and a fixed $\delta \in \F_{3^{2m}}$, if $(1-[\Trnn(\delta)]^4 )$ is a
square element in $\F_{3^m}$, then the polynomial
$$
F(X)=(X^{3^m}-X+\delta)^{3^{m}+4}+(X^{3^m}-X+\delta)^{5}+X
$$
is a permutation of $\F_{3^{2m}}$.
\end{lem}

We also recall some results providing us Walsh transform coefficients of some functions, which would be required later in our results. 
\begin{lem} \textup{\cite[Theorem 2]{AC}}
\label{L010}
 Let $K = \F_{2^k}$ and $f(X) = \Tr (X^{2^a +1} + X^{2^b +1})$, $0 \leq a < b$. Furthermore, let $d_1 = \gcd(b-a,k)$, $d_2 =\gcd(b+a,k)$ and $\nu = \max\{v_2(b-a),v_2(b+a)\}$, where $v_2$ is the $2$-valuation, that is, the largest power of $2$ dividing the argument. Also, let $S_{d_i} = \{X \in K : \Tr_{K/\F_{2^{d_i}}} (X) = 0\}$ and $L_i = X+X^{2^{d_i}}+X^{2^{2d_i}}+\cdots+X^{2^{\frac{k}{2}-d_i}}$,  for $i=1,2$. Then we have the following cases$:$\\
\textbf{Case $1$}$:$ $v_2(b-a) = v_2(b+a) = v_2(k)-1$ does not hold. Then$:$
\begin{enumerate}
 \item[$(1)$] $v_2(k) \leq \nu$. Then $\mathcal{W}_f(\alpha)=0$ if $\alpha \not \in S_{d_1} \cap S_{d_2}$.
 \item[$(2)$] $v_2(k) > \nu$. Then  $\mathcal{W}_{f}(\alpha)= 0$ if $(L_1 \circ L_2) \left(\alpha\right)\neq 0$.
\end{enumerate}
\textbf{Case $2$}$:$ $v_2(b-a) = v_2(b+a) = v_2(k)-1$. Then $\mathcal{W}_f(\alpha)=0$ if $(L_1 \circ L_2) \left(\alpha+1\right)\neq 0$.
\end{lem}
Further, we will also need the following lemma given in~\cite{AC} to evaluate $\mathcal{W}_{f_u}(\alpha)$ for any $\alpha \in K$, where $f_u$ is a more general function defined as
$f_u (x) = \Tr\left(uX^{2^a +1} + uX^{2^b +1}\right)$, with $u \in \F_{2^{d_1}}$, where $d_1 = \gcd(b-a,k)$.

\begin{lem} \textup{\cite[Theorem 3]{AC}}
 \label{L011}
 Let $K = \F_{2^k}$ and $f(X) = \Tr (X^{2^a +1} + X^{2^b +1})$, $0 \leq a < b$.  Furthermore, let $d_1 = \gcd(b-a,k)$, $d_2 =\gcd(b+a,k)$ and $\nu = \max\{v_2(b-a),v_2(b+a)\}$. Also, let $S_{d_i} = \{X \in K : \Tr_{K/\F_{2^{d_i}}} (X) = 0\}$ and $L_i = X+X^{2^{d_i}}+X^{2^{2d_i}}+\cdots+X^{2^{\frac{k}{2}-d_i}}$  for $i=1,2$. Moreover, if $f_u (x) = \Tr\left(uX^{2^a +1} + uX^{2^b +1}\right)$, where $u \in  \F_{2^{d_1}}$, then we have the following cases$:$\\
\textbf{Case $1$}$:$ If $v_2(a) = v_2(b) < v_2(k)$ does not hold, then 
for any $u \in  \F_{2^{d_1}}$ there exists $\beta \in  \F_{2^{d_1}}$ such that $u = \beta^{2^a +1}$ and 
$$
\mathcal{W}_{f_{u}}(\alpha)= \mathcal{W}_f\left(\frac{\alpha}{\beta}\right).
$$ 
\textbf{Case $2$}$:$ If $v_2(a) = v_2(b) < v_2(k)$, then$:$
\begin{enumerate}
 \item[$(1)$] $v_2(k) \leq \nu$. If $\frac{\alpha}{u^{2^{-b}}} \not \in S_{d_1} \cap S_{d_2}$, then $\mathcal{W}_{f_{u}}(\alpha)= 0$.
  \item[$(2)$] $v_2(k) > \nu$. Then  $\mathcal{W}_{f_{u}}(\alpha)= 0$ if $(L_1 \circ L_2) \left(\frac{\alpha}{u^{2^{-b}}}\right)\neq 0$.
\end{enumerate}
\end{lem}

 The following lemma can be gleaned from the proof of ~\cite[Proposition 2]{TH}.
\begin{lem}\label{owalsh} 
Let $m$ be a positive integer, $n=2m$ and $p$ an odd prime.  Then the absolute square of the Walsh transform coefficient at $v \in \F_{p^n}$ of the function $f (X) = \Tr\left( \sum_{i=0}^{m} a_i X^{p^i+1}\right)$, $a_i \in \F_{p^n}$,  is given by
 \begin{equation*} \lvert \mathcal{W}_f(v) \rvert^2 =
  \begin{cases}
   p^{n+\ell} &~\mbox{if}~f(X)+\Tr(-vX)\equiv0~\text{on Ker}~(L)  \\
    0 &~\mbox{otherwise},
  \end{cases}
 \end{equation*}
 where $\ell$ is dimension of kernel of the linearized polynomial $L(X) = \sum_{i=0}^{m} \left(a_i X^{p^i} +(a_i X)^{p^{n-i}}\right).$ 
\end{lem}

For our first theorem, we will be using the following lemma that we will now prove. The result for the case of finite fields with odd characteristic is already mentioned in the above Lemma~\ref{owalsh}. We show that this result also holds in the case of $p=2$ and we add its proof here for completeness.

\begin{lem} \label{walsh} Let $m$ be a positive integer and $n=2m$.  Then the square of the Walsh transform coefficient at $w \in \F_{2^n}$  of the function $f (X)=  \Tr\left( \sum_{i=0}^{m} a_i X^{2^i+1}\right)$,  $a_i \in \F_{2^n}$, is given by
 \begin{equation*} \mathcal{W}_f(w) ^2 =
  \begin{cases}
   2^{n+\ell} &~\mbox{if}~f(X)+\Tr(wX)\equiv0~\text{on Ker}~(L)  \\
    0 &~\mbox{otherwise},
  \end{cases}
 \end{equation*}
 where $\ell$ is dimension of kernel of the linearized polynomial $L(X) = \sum_{i=0}^{m} (a_i X^{2^i} +a_i^{2^{n-i}} X^{2^{n-i}}).$ 
\end{lem}
\begin{proof}
 We can easily write the square of the Walsh transform coefficient of the function $f : X \mapsto \Tr\left( \sum_{i=0}^{m} a_i X^{2^i+1}\right)$ at $w$ as 
 \allowdisplaybreaks\begin{align*}
  \mathcal{W}_f(w) ^2 & = \sum_{X,Y} (-1)^{f(X)+\Tr(wX)} (-1)^{f(Y)+\Tr(wY)} \\
  & = \sum_{Y,Z} (-1)^{f(Y+Z)+\Tr(w(Y+Z))} (-1)^{f(Y)+\Tr(wY)} \ (\text{where} \ X=Y+Z) \\
  & = \sum_{Z} (-1)^{f(Z)+\Tr(wZ)} \sum_{Y} (-1)^{f(Y)+f(Z)+f(Y+Z)}.
 \end{align*}
 
 We first simplify $f(Y)+f(Z)+f(Y+Z)$ as follows:
 \allowdisplaybreaks\begin{align*}
  f(Y)+f(Z)+f(Y+Z) & = \Tr\left( \sum_{i=0}^{m} a_i \left(Y^{2^i+1}+Z^{2^i+1}+(Y+Z)^{2^i+1}\right)\right)\\
  & = \Tr\left( \sum_{i=0}^{m} a_i \left(YZ^{2^i}+ZY^{2^i}\right)\right)\\
   & = \Tr\left( Y \sum_{i=0}^{m} \left(a_i Z^{2^i}+ (a_i Z)^{2^{n-i}} \right)\right)\\
   & = \Tr(YL(Z)),
 \end{align*}
where $L(Z)=\sum_{i=0}^{m} (a_i Z^{2^i}+ a_i^{2^{n-i}} Z^{2^{n-i}})$ (from Lemma~\ref{walsh}) is the linearized polynomial over $\F_{2^n}$ with kernel of dimension $l$, denoted by Ker($L$). This will give us 
\allowdisplaybreaks\begin{align*}
  \mathcal{W}_f(w) ^2 & = \sum_{Z} (-1)^{f(Z)+\Tr(wZ)} \sum_{Y} (-1)^{\Tr(YL(Z))}\\
  & = 2^n \sum_{Z \in Ker(L)} (-1)^{f(Z)+\Tr(wZ)}.
 \end{align*}
The above equality holds because for those $Z \not \in$ Ker ($L$), $YL(Z)$ forms a permutation over $\F_{2^n}$ making the inner sum in the square of the Walsh transform zero. To proceed further, we consider $\F_{2^n}$ as an $n$-dimensional vector space over $\F_{2}$ and hence $L(Z)$ as a linear transformation of $\F_{2^n}$. As $f(Y)+f(Z)+f(Y+Z)=\Tr(YL(Z))$, we get that $f(Z)+ \Tr(wZ)$ is linear on the kernel of $L$. This implies that either $f(Z)+ \Tr(wZ)$ is identically zero on Ker($L$), or $f(Z)+ \Tr(wZ)$ is an onto map on Ker($L$). Hence, the claim is shown.
\end{proof}

Next, we recall the general technique given in~\cite{StanicaPG2021} using the expression for the number of solutions to a given equation over finite fields in terms of Weil sums. The authors used this technique to compute the $c$-DDT entries. Let $\chi_1: \F_q \rightarrow \mathbb{C}$ be the canonical additive character of the additive group of $\F_q$ defined as follows
\[
 \chi_1(X):= \exp \left(\frac{2\pi i \Tr(X)}{p} \right).
\]
One can easily observe (see, for instance~\cite{PSweil}) that number of solutions $(X_1, X_2, \ldots, X_n) \in \F_q^n$ of the equation $F(X_1, X_2, \ldots, X_n)=b,$ denoted by $N(b)$, is given by 
\begin{equation}\label{ddtw}
 \begin{split}
 N(b)= \frac{1}{q} \sum_{X_1,X_2, \ldots, X_n \in \F_q} \sum_{\beta \in \F_q} \chi_1(\beta(F(X_1, X_2, \ldots, X_n)-b)).
 \end{split}
\end{equation}
We will be using the above expression to calculate the $c$-differential uniformity of a few permutations over finite fields in the forthcoming sections.

\section{Permutations over $\F_{2^n}$ with low $c$-differential uniformity}\label{S3}

 We first consider the $c$-differential uniformity of the function $F(X)=(X^{2^m}+X+\delta)^{2^{2m-2}+2^{m-2}+1}+X$ over $\F_{2^n}$, where $n=2m$ and $\delta \in \F_{2^n}$. From Lemma~\ref{L02}, we know that $F$ is a permutation polynomial over $\F_{2^n}$. In the following theorem, we give conditions on $\delta$ and $c$ for which $F$ turns out to be either a P$c$N or an AP$c$N function.

\begin{thm} \label{T1}
Let $F(X)=(X^{2^m}+X+\delta)^{2^{2m-2}+2^{m-2}+1}+X$ over $\F_{2^{n}}$, where $n=2m$.
\begin{enumerate}
 \item[$(1)$] Let $\delta\in\F_{2^m}$. Then $F$ is P$c$N for $c \in \F_{2^m}\setminus \{1\}$. Moreover, $F$ is AP$c$N for $c \in \F_{2^n}\setminus \F_{2^m}$.
\item[$(2)$] Let $\delta \in \F_{2^n}\setminus \F_{2^m}$ with $\Trtm(\delta) = \Trm(1)$ and $p_m((\delta + \bar \delta)^{-1}) \neq 0$, where $p_m(X)$ is the polynomial defined in Lemma~\textup{\ref{L01}}. Then $F$ is P$c$N, if $c \in \F_{2^m}\setminus \{1\}$ and is of $c$-differential uniformity less than or equal to $4$, if $c \in \F_{2^n}\setminus \F_{2^m}$.
\end{enumerate}
\end{thm}
\begin{proof}
Clearly, by expanding the given trinomial, one can easily simplify $F(X)$ and gets the below expression
\allowdisplaybreaks\begin{align*}
F(X) & = \Trnn(X^{2^{m-1}+1} +X^{2^{2m-1}+1}) +\delta \Trnn(X^{2^{m-1}}) +(\delta^{2^{m-2}+1}+\delta^{2^{2m-2}+1}) \Trnn(X^{2^{m-2}}) \\
& \qquad + \Trnn(\delta^{2^{m-2}}) \Trnn(X^{2^{m-2}+1}+X^{2^{2m-2}+1}) + \delta^{2^{m-2}+2^{2m-2}} \Trnn(X) \\
& \qquad \qquad \qquad + X + \delta^{2^{2m-2}+2^{m-2}+1}.
\end{align*}
Recall that for any $(a, b) \in \F_{2^{n}} \times \F_{2^{n}}$,
the $c$-DDT entry $_c\Delta_F(a, b)$ is given by the number of solutions $X \in \F_{2^{n}}$ of the following equation
\begin{equation}\label{eq1}
 F(X+a)+cF(X)=b,
\end{equation}
which is the same as
\allowdisplaybreaks\begin{align*}
 & (1+c)F(X)+ \Trnn((a^{2^{m-1}}+a^{2^{2m-1}})X+a(X^{2^{m-1}}+X^{2^{2m-1}})) +F(a)+ \delta^{2^{2m-2}+2^{m-2}+1} \\
 &  \qquad+ \Trnn(\delta^{2^{m-2}}) \Trnn((a^{2^{m-2}}+a^{2^{2m-2}})X+a(X^{2^{m-2}}+X^{2^{2m-2}})) = b.
\end{align*}
 Now, by using Equation~\eqref{ddtw}, the number of solutions $X \in \F_{2^n}$ of the above equation, that is, $_c\Delta_F(a,b)$, is given by 
\allowdisplaybreaks\begin{align*}
2^n \  _c\Delta_F(a,b) & = 
 \sum_{\beta \in \F_{2^n}} \sum_{X \in \F_{2^n}} (-1)^{\displaystyle{\Tr\left(\beta(1+c)F(X)\right)}} \\
 & \qquad (-1)^{\displaystyle{\Tr\left(\beta\left(\Trnn\left(\left(a^{2^{m-1}}+a^{2^{2m-1}}\right)X+a\left(X^{2^{m-1}}+X^{2^{2m-1}}\right)\right)+b\right)\right)}} \\
 & \qquad (-1)^{\displaystyle{\Tr\left(\beta\left(\Trnn\left(\delta^{2^{m-2}}\right) \Trnn\left(\left(a^{2^{m-2}}+a^{2^{2m-2}}\right)X+a\left(X^{2^{m-2}}+X^{2^{2m-2}}\right)\right)\right)\right)}}\\
& \qquad \qquad \qquad(-1)^{\displaystyle{\Tr\left(\beta\left(F(a)+ \delta^{2^{2m-2}+2^{m-2}+1} \right)\right)}},
\end{align*}
that is,
\begin{equation*}
 \begin{split}
 2^n \ _c\Delta_F(a,b) & =  \sum_{\beta \in \F_{2^n}} (-1)^{\displaystyle{\Tr\left(\beta\left(F(a)+b+\delta^{2^{2m-2}+2^{m-2}+1} \right)\right)}} \sum_{X \in \F_{2^n}} (-1)^{\displaystyle{\Tr\left(\beta(1+c)F(X)\right)}} \\
  &   (-1)^{\displaystyle{\Tr\left(\beta\left(\Trnn\left(\left(a^{2^{m-1}}+a^{2^{2m-1}}\right)X+a\left(X^{2^{m-1}}+X^{2^{2m-1}}\right)\right)\right)\right)}} \\
  & \quad(-1)^{\displaystyle{\Tr\left(\beta\left( \Trnn\left(\delta^{2^{m-2}}\right) \Trnn\left(\left(a^{2^{m-2}}+a^{2^{2m-2}}\right)X+a\left(X^{2^{m-2}}+X^{2^{2m-2}}\right)\right)\right)\right)}}.
 \end{split}
\end{equation*}
We now use the following notations
\allowdisplaybreaks\begin{align*}
 T_0 & = \Tr\left(\beta(1+c)F(X)\right) \\
 T_1 & = \Tr\left(\beta\left( \Trnn\left(\left(a^{2^{m-1}}+a^{2^{2m-1}}\right)X+a\left(X^{2^{m-1}}+X^{2^{2m-1}}\right)\right)\right)\right) \\
 & \qquad \qquad +\Tr\left(\beta\left(\Trnn\left(\delta^{2^{m-2}}\right) \Trnn\left(\left(a^{2^{m-2}}+a^{2^{2m-2}}\right)X+a\left(X^{2^{m-2}}+X^{2^{2m-2}}\right)\right)\right)\right).
\end{align*}
Therefore,
\begin{equation}
\label{eq2}
 _c\Delta_F(a,b) = \frac{1}{2^n} \sum_{\beta \in \F_{2^n}} (-1)^{\displaystyle{\Tr\left(\beta\left(F(a)+b+\delta^{2^{2m-2}+2^{m-2}+1}\right)\right)}}\sum_{X \in \F_{2^n}} (-1)^{\displaystyle{T_0+T_1}}.
\end{equation}

\textbf{Case 1.} Let $c\in \F_{2^n}$ and $\delta \in \F_{2^m}$. To compute $T_0$ and $T_1$, we first write
\allowdisplaybreaks
\begin{equation*} 
\begin{split}
T_1 & = \Tr\left(\beta\left( \Trnn\left(\left(a^{2^{m-1}}+a^{2^{2m-1}}\right)X+a\left(X^{2^{m-1}}+X^{2^{2m-1}}\right)\right)\right)\right) \\
& = \Tr\left(\left(\Trnn(a^{2^{m-1}}\right)\Trnn(\beta) + \Trnn\left(a^{2}\right)\Trnn\left(\beta^2\right)X\right)
\end{split}
\end{equation*}
and
 \allowdisplaybreaks 
 \begin{equation*} 
\begin{split}
T_0 & = \Tr(\beta(1+c)F(X)) \\
& = \Tr\left(\beta(1+c)\left(\Trnn\left(X^{2^{m-1}+1} +X^{2^{2m-1}+1}\right) +\delta \Trnn\left(X^{2^{m-1}}\right) + X + \delta^{2^{m-1}+1}\right) \right)\\
& \qquad \qquad + \Tr\left(\beta(1+c) \delta^{2^{m-1}} \Trnn(X) \right)\\
& = \Tr(\beta(1+c)\delta^{2^{m-1}+1})+ \Tr\left(\Trnn\left(\beta(1+c)\right)\left(X^{2^{m-1}+1} +X^{2^{2m-1}}\right)\right) \\
& \qquad + \Tr\left((\beta(1+c) + \Trnn\left( \delta^2\beta^2 (1+c)^2 \right)+ \Trnn(\delta^{2^{m-1}}\beta(1+c))) X\right).
\end{split}
\end{equation*}
Now Equation~\eqref{eq2} reduces to
\allowdisplaybreaks\begin{align*}
 _c\Delta_F(a,b) & = \frac{1}{2^n} \sum_{\beta \in \F_{2^n}} (-1)^{\displaystyle{\Tr\left(\beta(F(a)+b+c\delta^{2^{m-1}+1})\right)}} \\
 & \qquad \qquad \sum_{X \in \F_{2^n}} (-1)^{\displaystyle{\Tr\left(u(X^{2^{m-1}+1}+X^{2^{2m-1}+1})+wX\right)}},
\end{align*}
where
\allowdisplaybreaks
\begin{align*}
 u & = \Trnn((1+c)\beta) \\
 w & =  \Trnn\left(a^{2^{m-1}}\right)\Trnn(\beta) + \Trnn(a^{2})\Trnn\left(\beta^2\right)\\
 & \qquad +\beta(1+c) + \delta^2\Trnn\left((1+c)^2 \beta^2\right) + \delta^{2^{m-1}}\Trnn\left((1+c)\beta\right).
\end{align*}
Further, splitting the above sum depending on whether $\Trnn((1+c)\beta)$ is~$0$ or not, we get
\allowdisplaybreaks
\begin{align*}
 2^n \ _c\Delta_F(a,b) & = \sum_{\substack{\beta \in \F_{2^n} \\ \Trnn(\beta(1+c))=0}} (-1)^{\displaystyle{\Tr\left(\beta\left(F(a)+b+c\delta^{2^{m-1}+1}\right)\right)}} \sum_{X \in \F_{2^n}} (-1)^{\displaystyle{\Tr\left(\left(\beta(1+c)\right)X\right)}} \\
 & \qquad \qquad (-1)^{\displaystyle{\Tr\left(\left(\Trnn\left(a^{2^{m-1}}\right)\Trnn(\beta) + \Trnn(a^{2})\Trnn(\beta^2)\right)X \right)}}\\
 & +  \sum_{\substack{\beta \in \F_{2^n} \\ \Trnn(\beta(1+c))\neq 0}} (-1)^{\displaystyle{\Tr\left(\beta(F(a)+b+c\delta^{2^{m-1}+1})\right)}}\\
 & \qquad \qquad \qquad \sum_{X \in \F_{2^n}} (-1)^{\displaystyle{\Tr\left(u(X^{2^{m-1}+1}+X^{2^{2m-1}+1})+wX\right)}}\\
 &= S_0 + S_1,
\end{align*}
where $S_0,S_1$ are the two inner sums. We first consider the sum $S_0$ below,
\allowdisplaybreaks
\begin{align*}
 S_0 & = \sum_{\substack{\beta \in \F_{2^n} \\ \Trnn(\beta(1+c))=0}} (-1)^{\displaystyle{\Tr\left(\beta\left(F(a)+b+c\delta^{2^{m-1}+1}\right)\right)}}
\sum_{X \in \F_{2^n}} (-1)^{\displaystyle{\Tr\left(\left(\beta(1+c)\right)X\right)}} \\
& \qquad \qquad \qquad (-1)^{\displaystyle{\Tr\left( \left(\Trnn\left(a^{2^{m-1}}\right)\Trnn(\beta) + \Trnn(a^{2})\Trnn(\beta^2)\right)X\right)}}.
\end{align*}
To compute $S_0$, we need to find the number of solutions of the following equation in $\F_{2^n}$,
\begin{equation}\label{eq}
 (1+c)\beta + \Trnn(a^{2^{m-1}})\Trnn(\beta) + \Trnn(a^{2})\Trnn(\beta^2)=0.
\end{equation}
If $c \in \F_{2^m} \setminus \{1\}$, then   $\Trnn((1+c)\beta)=(1+c)\Trnn(\beta)=0$ and hence  Equation~\eqref{eq} reduces to $(1+c)\beta=0$. Thus, the inner sum in $S_0$ is zero for all $\beta \in \F_{2^n}^{*}$ and so $S_0=2^n$.

Next, we let $c \in \F_{2^n}\setminus \F_{2^m}$. Then we observe that for $\beta$'s satisfying $\Trnn((1+c)\beta)=0$, we have $\beta^{2^m}= \tilde c \beta$, where $\tilde c = (1-c)^{1-2^m}$. If $\Trnn(a)=0$, then Equation~\eqref{eq} vanishes only for $\beta=0$. Now, we assume that $\Trnn(a) \neq 0$. Then, after substituting the value of $\beta^{2^m}$ in Equation~\eqref{eq}, it follows that  Equation~\eqref{eq} can have at most two solutions, namely, $\beta_1=0$ and $\beta_2 = \dfrac{ \Trnn(a^{2^{m-1}})(1+\tilde c)+(1+c)}{\Trnn(a^2)(1+\tilde c)^2}$, which is clearly nonzero for some $a \in \F_{2^n}$ with $\Trnn(a^{2^{m-1}}) \neq \dfrac{1+c}{1+\tilde c}$. Hence, we get that the inner sum in $S_0$ is zero for all $\beta \in \F_{2^n} \setminus \{\beta_1,\beta_2\}$. Thus, for $(a,b) = (a, F(a) + c\delta^{2^{m-1}+1})$ along with $\Trnn(a^{2^{m-1}}) \neq \dfrac{1+c}{1+\tilde c}$, we have $S_0 = 2^{n+1}$;  otherwise, $S_0 = 2^n$.

Next, we consider
\allowdisplaybreaks
\begin{align*}
 S_1 & = \sum_{\substack{\beta \in \F_{2^n} \\ \Trnn((1+c)\beta)\neq 0}} (-1)^{\displaystyle{\Tr\left(\beta\left(F(a)+b+c\delta^{2^{m-1}+1}\right)\right)}} \\
& \qquad \qquad \sum_{X \in \F_{2^n}} (-1)^{\displaystyle{\Tr\left(u\left(X^{2^{m-1}+1}+X^{2^{2m-1}+1}\right)+wX\right)}} \\
 & = \sum_{\substack{\beta \in \F_{2^n} \\ \Trnn((1+c)\beta)\neq 0}} (-1)^{\displaystyle{\Tr\left(\beta\left(F(a)+b+c\delta^{2^{m-1}+1}\right)\right)}}
\sum_{X \in \F_{2^n}} \mathcal{W}_{f_u}(w),
\end{align*}
where  $\mathcal{W}_{f_u}(w)$ is the Walsh transform of the function $f_u (X)=   \Tr(u(X^{2^{m-1}+1}+X^{2^{2m-1}+1}))$ at $w$.  
We split our analysis into three subcases depending on the value of $m$. Note that we have $k=2m$, $a =m-1$ and $b=2m-1$.

\textbf{Subcase 1(a).} Let $m$ be odd. Then $\gcd(b-a,k)=m$ and $\gcd(b+a,k)=1$. Also, notice that $v_2(a)= v_2(m-1)$, $v_2(b) = v_2(2m-1) = 0$ and $v_2(k)=v_2(2m)=1$. Summarizing, we observe that $v_2(a) = v_2(b) \leq v_2(k)$ does not hold. From Lemma~\ref{L011}, one can see that $\mathcal{W}_{f_{u}}(w)= \mathcal{W}_f(\frac{w}{\eta})$, where $f(X) =  \Tr(X^{2^{m-1}+1}+X^{2^{2m-1}+1})$ and $\eta \in \F_{2^m}$ such that $u = \eta^{2^{m-1}+1}$. We now show that $\mathcal{W}_f(\frac{w}{\eta})=0$ using Lemma~\ref{L010}. One can clearly observe that $0=v_2(b-a)=v_2(b+a)=v_2(k)-1$. Hence, it is sufficient to show $(L_1 \circ L_2)(\frac{w}{\eta}+1)\neq 0.$ Now,
\[
(L_1 \circ L_2)(X) = X+ X^2 +X^{2^2} + \cdots + X^{2^{m-1}}.
\]
Hence, 
\allowdisplaybreaks\begin{align*}
 (L_1 \circ L_2)\left(\frac{w}{\eta}+1\right) & =m+ \frac{\beta(1+c)}{\eta} + \left(\frac{\beta(1+c)}{\eta}\right)^2 +\cdots +\left(\frac{\beta(1+c)}{\eta}\right)^{2^{m-1}}\\
 & + \Trm\left( \frac{\Trnn(a^{2^{m-1}})\Trnn(\beta) + \Trnn(a)^2\Trnn(\beta)^2}{\eta}\right) \\
 & +\Trm\left( \frac{\delta^2(1+c)^2 \Trnn(\beta)^2 + \delta^{2^{m-1}}(1+c)\Trnn(\beta)}{\eta}\right).
\end{align*}
Now, if $(L_1 \circ L_2)\left(\frac{w}{\eta}+1\right)=0$, then we have $\left((L_1 \circ L_2)\left(\frac{w}{\eta}+1\right)\right)^2=0$. This will give us 
\begin{equation*}
\begin{split}
  (L_1 \circ L_2)\left(\frac{w}{\eta}+1\right)+ \left((L_1 \circ L_2)\left(\frac{w}{\eta}+1\right)\right)^2 & =\frac{1}{\eta} \Trnn(\beta(1+c)) +m +m^2\\
  & =\frac{u}{\eta} + m+m^2 = 0.
  \end{split}
\end{equation*}
As $u^2 = \eta^3$, we have $u = (m+m^2)^3 \equiv 0 \pmod 2$, which gives us $\Trnn((1+c)\beta) = 0$. This is obviously not true, and  the claim is shown.

\textbf{Subcase 1(b).} Let $m \equiv 0 \pmod 4$. If we have $v_2(m)=v ( \geq 2)$, then $v_2(a)= v_2(b) = 0$ and $v_2(k)=v_2(2m)=v+1$. Summarizing, we have $v_2(a)= v_2(b) < v_2(k)$ and $v_2(k) > \nu$. Thus, by using Lemma~\ref{L011} if we show that $L_1 \circ L_2 \left(\frac{w}{u^2}\right) \neq 0$ then we have $\mathcal{W}_G(w)=0$. Therefore, $S_1 =0$ and the claim is proved. 

Our goal is to now show $L_1 \circ L_2 \left(\frac{w}{u^2}\right)  \neq 0$. As $d_1 =m$ and $d_2=2$, we have 
\[
(L_1 \circ L_2)(X) = X+ X^{2^2} +X^{2^{(2\cdot2)}} + X^{2^{(3\cdot2)}}+ \cdots + X^{2^{m-2}}.
\]
Now, if $(L_1 \circ L_2)\left(\frac{w}{u^2}\right)=0$, then we have $\left((L_1 \circ L_2)\left(\frac{w}{u^2}\right)\right)^{2^2}=0$. This will give us 
\begin{equation*}
  (L_1 \circ L_2)\left(\frac{w}{u^2}\right) + \left((L_1 \circ L_2)\left(\frac{w}{u^2}\right)\right)^{2^2} = \Trnn\left(\frac{w}{u^2}\right)= \frac{1}{u^2} \Trnn(w)  =0,
\end{equation*}
which is a contradiction to the assumption that $\Trnn(\beta(1+c))\neq0$.

\textbf{Subcase 1(c).} Let $m \equiv 2 \pmod 4$, then $d_1=m$ and $d_2= 4$. In this subcase, we have $v_2(a)= v_2(b) = 0$ and $v_2(k)=v_2(2m)=2$. Summarizing all, we observe that $v_2(a) = v_2(b) < v_2(k)$ holds and $v_2(k) \leq \nu$. Again using Lemma~\ref{L011}, if we show that $ \frac{w}{u^2} \not \in S_{d_1} \cap S_{d_2}$, then we are done. In this subcase, $S_{d_1} = \{X \in \F_{2^n} : \Trnn(X) = 0\}$ and $ S_{d_2} = \{X \in \F_{2^n} : \Trfn(X) = 0\}$.
Let $ \frac{w}{u^2} \in S_{d_1} \cap S_{d_2}$. Then we have $\Trnn\left( \frac{w}{u^2}\right) = 0$, which implies that $\Trnn(\beta(1+c)) = 0$, which is not possible. Hence, $\mathcal{W}_G(w)=0$, which renders $S_1 =0$.

This above analysis shows the claim that $F$ is P$c$N for $c \in \F_{2^m}\setminus \{1\}$ and AP$c$N for $c \in \F_{2^n}\setminus \F_{2^m}$.

\textbf{Case 2.} Let $\delta \in \F_{2^n} \setminus  \F_{2^m}$ with $\Trtm(\delta) = \Trm(1)$ and $p_m((\delta + \bar \delta)^{-1}) \neq 0$. Then $T_0,T_1$ become
\allowdisplaybreaks
\begin{align*}
T_1 & = \Tr(\beta( \Trnn((a^{2^{m-1}}+a^{2^{2m-1}})X+a(X^{2^{m-1}}+X^{2^{2m-1}})))) \\
 & \qquad \qquad +\Tr(\beta(\Trnn(\delta^{2^{m-2}}) \Trnn((a^{2^{m-2}}+a^{2^{2m-2}})X+a(X^{2^{m-2}}+X^{2^{2m-2}}))))\\
& = \Tr((\Trnn(a^{2^{m-1}})\Trnn(\beta) + \Trnn(a^2)\Trnn(\beta^2))X)\\
& \qquad \qquad + \Tr((\Trnn(\delta^{2^{m-2}})\Trnn(a^{2^{m-2}})\Trnn(\beta) + \Trnn(\delta)\Trnn(a^{2^2})\Trnn(\beta)^{2^2})X),
\end{align*}
and
\allowdisplaybreaks
\begin{align*}
T_0 & = \Tr(\beta(1+c)F(X)) \\
& = \Tr\left(\beta(1+c)(\Trnn(X^{2^{m-1}+1} +X^{2^{2m-1}+1}) +\delta \Trnn(X^{2^{m-1}}) + X + \delta^{2^{2m-2}+2^{m-2}+1}) \right)\\
& \quad + \Tr\left(\beta(1+c) (\delta^{2^{m-2}+2^{2m-2}} \Trnn(X) + \Trnn(\delta^{2^{m-2}}) \Trnn(X^{2^{m-2}+1}+X^{2^{2m-2}+1}) ) \right)\\
& \qquad \qquad \qquad +\Tr\left(\beta(1+c) (\delta^{2^{m-2}+1}+\delta^{2^{2m-2}+1}) \Trnn(X^{2^{m-2}})\right)\\
& = \Tr\left( \Trnn(\beta(1+c)) (X^{2^{m-1}+1}+ \Trnn(\delta^{2^{m-2}}) X^{2^{m-2}+1}) \right)\\
& \qquad + \Tr\left(\Trnn(\beta(1+c))^2 X^{2^{1}+1} +  \Trnn(\delta)\Trnn(\beta(1+c))^{2^2} X^{2^{2}+1}))\right) \\
& \qquad \qquad + \Tr\left((\Trnn(\delta^2 \beta^2(1+c)^2) + \beta(1+c) + \Trnn(\beta(1+c)(\delta^{2^{m-2}+2^{2m-2}})))X\right)\\
& \qquad + \Tr\left(\Trnn(\beta^{2^2}(1+c)^{2^2}(\delta^{2^{m-2}+1}+\delta^{2^{2m-2}+1})^{2^2}) X\right) + \Tr\left(\beta(1+c)\delta^{2^{2m-2}+2^{m-2}+1}\right).
\end{align*}
Thus, we can rewrite Equation~\eqref{eq2} as
\begin{equation*}
\begin{split}
 _c\Delta_F(a,b) & = \frac{1}{2^n} \sum_{\beta \in \F_{2^n}} (-1)^{\displaystyle{\Tr\left(\beta(F(a)+b+c\delta^{2^{2m-2}+2^{m-2}+1})\right)}} \\
 & \sum_{X \in \F_{2^n}} (-1)^{\displaystyle{\Tr(u_1X^{2^{m-1}+1}+u_2 X^{2^{1}+1} +u_3 X^{2^{m-2}+1}+u_4 X^{2^{2}+1}+wX)}},
\end{split}
\end{equation*}
where 
\allowdisplaybreaks\begin{align*}
 u_1 & = \Trnn((1+c)\beta) \\
 u_2 & = \Trnn((1+c)\beta)^2 = u_1^2\\
 u_3 & = \Trnn(\delta^{2^{m-2}}) \Trnn((1+c)\beta) \\
 u_4 & = \Trnn(\delta) \Trnn((1+c)\beta)^{2^2} = u_3^{2^2}\\
 w & = \Trnn(\delta^2 \beta^2(1+c)^2) + \beta(1+c) + \Trnn(\beta(1+c)(\delta^{2^{m-2}+2^{2m-2}})) \\ 
 & +\Trnn(\beta^{2^2}(1+c)^{2^2}(\delta^{2^{m-2}+1}+\delta^{2^{2m-2}+1})^{2^2} )+ \Trnn(\delta)\Trnn(a^{2^2})\Trnn(\beta)^{2^2} \\
 & +\Trnn(a^{2^{m-1}})\Trnn(\beta)  + \Trnn(a^2)\Trnn(\beta^2)+\Trnn(\delta^{2^{m-2}})\Trnn(a^{2^{m-2}})\Trnn(\beta).
\end{align*}

We now split our analysis into two cases and define $S_0$ and $S_1$ depending on whether $\Trnn((1+c)\beta)=0$ or not, respectively. We first compute $S_0$ as follows:
\allowdisplaybreaks
\begin{align*}
 S_0 & = \sum_{\substack{\beta \in \F_{2^n} \\ \Trnn(\beta(1+c))=0}} (-1)^{\displaystyle{\Tr\left(\beta(F(a)+b+c\delta^{2^{2m-2}+2^{m-2}+1})\right)}}
\sum_{X \in \F_{2^n}} (-1)^{\displaystyle{\Tr\left(\left(\beta(1+c)\right)X\right)}} \\
&  (-1)^{\displaystyle{\Tr\left( (\Trnn(\beta^{2^2}(1+c)^{2^2}(\delta^{2^{m-2}+1}+\delta^{2^{2m-2}+1})^{2^2}) + \Trnn(a^{2^{m-1}})\Trnn(\beta))X\right)}}\\
&  (-1)^{\displaystyle{\Tr\left((\Trnn(\delta^{2^{m-2}})\Trnn(a^{2^{m-2}})\Trnn(\beta) + \Trnn(\delta)\Trnn(a^{2^2})\Trnn(\beta)^{2^2})X\right)}} \\
&  (-1)^{\displaystyle{\Tr\left((\Trnn(\delta^2 \beta^2(1+c)^2) + \Trnn(\beta(1+c)(\delta^{2^{m-2}+2^{2m-2}}))+ \Trnn(a^2)\Trnn(\beta^2))X\right)}}.
\end{align*}

If $c \in \F_{2^m}$, then $\Trnn((1+c)\beta) =0$ implies $\Trnn(\beta)=0$ and thus $\beta^{2^m} =\beta$. This would further reduce $S_0$ as follows:
\allowdisplaybreaks
\begin{align*}
 S_0 & = \sum_{\substack{\beta \in \F_{2^n} \\ \Trnn(\beta(1+c))=0}} (-1)^{\displaystyle{\Tr\left(\beta(F(a)+b+c\delta^{2^{2m-2}+2^{m-2}+1})\right)}}
\sum_{X \in \F_{2^n}} (-1)^{\displaystyle{\Tr\left(\left(\beta(1+c)\right)X\right)}} \\
& \qquad (-1)^{\displaystyle{\Tr\left( (\Trnn((\delta^{2^{m-2}+1}+\delta^{2^{2m-2}+1})^{2^2} )\beta^{2^2}(1+c)^{2^2}+ \Trnn(\delta^2)\beta^2(1+c)^2)X\right)}}.
\end{align*}
To compute $S_0$, we need to find the number of solutions $   \beta\in\F_{2^n}$ of the following equation 
$$
\Trnn\left((\delta^{2^{m-2}+1}+\delta^{2^{2m-2}+1})^{2^2} \right)(1+c)^{4}\beta^{4}+ \Trnn(\delta^2)(1+c)^2\beta^2+(1+c) \beta =0,
$$
or equivalently, to find the number of solutions $\beta\in\F_{2^m}^{*}$ of the equation given below,
$$
\Trnn\left(\delta^{2^{m}+2^2}+\delta^{1+2^2}\right)(1+c)^{3}\beta^{3}+ \Trnn(\delta^2)(1+c)\beta+1 =0.
$$
Further, multiplying the above equation by $\Trnn(\delta)$ and using $Z = \Trnn(\delta^2)(1+c)\beta$, one can  rewrite the above equation as follows,
$$
Z^3+ \Trnn(\delta^{2^{m-1}})^2 Z+\Trnn(\delta^{2^{m-1}})^2=0.
$$
Substituting $Z$ with $\Trnn(\delta^{2^{m-1}})Z$ in the above equation, we obtain $Z^3+ Z+\Trnn(\delta^{2^{m-1}})^{-1}=0.$ From Lemma~\ref{L01}, the cubic equation $Z^3+ Z+\Trnn(\delta^{2^{m-1}})^{-1}=0$ cannot have a unique solution because $\Trm( \Trnn(\delta^{2^{m-1}})+1)=\Trm( \Trnn(\delta)+1)=\Trn(\delta) + \Trm(1)=0$. Let $p_m(\Trnn(\delta^{2^{m-1}})^{-1})=0$. Then the cubic equation has three distinct solutions in $\F_{2^n}$ and hence the cubic equation $Z^3+ Z+\Trnn(\delta)^{-1}=0$ has three distinct solutions in $ \F_{2^n}$, implying that $p_m((\delta+\bar \delta)^{-1})=0$, which contradicts the given assumption. Hence $S_0=2^n$.  

Let $c \in \F_{2^n} \setminus \F_{2^m}$. Using a similar technique as above, we are led to finding the number of solutions $\beta \in \F_{2^n}$ of the following equation,
\begin{equation*}
 \begin{split}
  &  \Trnn(\beta^{2^2}(1+c)^{2^2}(\delta^{2^{m-2}+1}+\delta^{2^{2m-2}+1})^{2^2}) + \Trnn(a^{2^{m-1}})\Trnn(\beta) + \Trnn(a^2)\Trnn(\beta^2) \\
  & + \Trnn(\delta^{2^{m-2}})\Trnn(a^{2^{m-2}})\Trnn(\beta) + \Trnn(\delta)\Trnn(a^{2^2})\Trnn(\beta)^{2^2} + (1+c)\beta \\
  & \qquad \qquad  + \Trnn(\delta^2 \beta^2(1+c)^2) + \Trnn(\beta(1+c)(\delta^{2^{m-2}+2^{2m-2}})) =0.
 \end{split}
\end{equation*}

Further, we reduce the above equation by substituting $\beta^{2^m} = \dfrac{(1+c)}{(1+c)^{2^m}} \beta = \tilde c \beta$ , where $\tilde c = \dfrac{(1+c)}{(1+c)^{2^m}}$, to get 
\allowdisplaybreaks
\begin{align*}
  &  (1+c)^4 \Trnn(\delta^{2^m+2^2} +\delta^{2^2+1}) \beta^4 + (1+ \tilde c) \Trnn(a^{2^{m-1}}) \beta+ (1+ \tilde c)^2 \Trnn(a^{2}) \beta^2 \\
  & \quad + (1+ \tilde c) \Trnn(\delta^{2^{m-2}}) \Trnn(a^{2^{m-2}}) \beta + (1+ \tilde c)^4 \Trnn(\delta) \Trnn(a^{2^2}) \beta^4 +(1+c) \beta \\
  &\quad + (1+c)^2 \Trnn(\delta)^2 \beta^2 = 0.
 \end{align*}
Combining the coefficients of $\beta^i$ for $i =1,2$ and $4$, we rewrite the above equation in the following simplified way, $A(1+c)^4\beta^4 + B (1+c)^2\beta^2 + C(1+c)\beta=0$, where 
\begin{align*}
 A & = \Trnn(\delta)^5 +\frac{(1+ \tilde c)^4}{(1+c)^4} \Trnn(\delta) \Trnn(a^{2^2}),\\
 B & =\frac{(1+ \tilde c)^2}{(1+c)^2} \Trnn(a^{2}) + \Trnn(\delta)^2, \\
 C & = \frac{(1+ \tilde c)}{(1+c)} \Trnn(a^{2^{m-1}}) + \frac{(1+ \tilde c)}{(1+c)} \Trnn(\delta^{2^{m-2}}) \Trnn(a^{2^{m-2}})  + 1.
\end{align*}
Notice that $A = \Trnn(\delta) B^2$, and hence 
substituting $u=(1+c)\beta$, we obtain the equation $\Trnn(\delta) B^2 u^4 + B u^2 + C u=0$. Thus, it is sufficient to consider the solutions of the cubic equation $\Trnn(\delta) B^2 u^3 + B u + C =0$. Observe that $B$ and $C$ are in $\F_{2^m}$, which implies that the cubic equation $\Trnn(\delta) B^2 u^3 + B u + C =0$ is over $\F_{2^m}$. Hence our goal is to find the number of solutions of the above cubic equation in $\F_{2^m}$. Multiplying the cubic equation by $\Trnn(\delta^2) B$ and substituting $Z= (\Trnn(\delta)B) u$, we get $Z^3 + \Trnn(\delta) B Z + \Trnn(\delta)^2 BC =0$, or equivalently
$$
Z^3 + \Trnn(\delta^{2^{m-1}})^2 (B')^2 Z + \Trnn(\delta)^2 BC =0,
$$ 
where $(B')^2 = B$. Now replacing $Z$ by $\Trnn(\delta^{2^{m-1}})B'Z$, we have 
$$
 Z^3 +  Z + \frac{\Trnn(\delta)^2 BC}{\Trnn(\delta^{2^{m-1}})^3 (B')^3} =0,
$$
which simplifies to
\begin{equation} 
\label{new1}
Z^3 +  Z + \frac{\Trnn(\delta)^{2^{m-1}} C}{B'} =0.
\end{equation}
Notice that for $a \in \F_{2^m}$, the above equation reduces to $Z^3 +  Z + \frac{\Trnn(\delta)^{2^{m-1}}}{\Trnn(\delta)} =0$. If the reduced equation has three solutions in $\F_{2^m}$ then $Z^3 +  Z + \frac{1}{\Trnn(\delta)} =0$ also has three distinct solutions which is not true as $p_m((\delta + \bar \delta)^{-1}) \neq 0$. Also, $\Trm( \Trnn(\delta)+1) \neq 1$. Thus, for $a \in \F_{2^m}$, there are no solutions to Equation~\eqref{new1} in $\F_{2^m}$. If $a \not \in \F_{2^m}$, then we may have at most three solutions to Equation~\eqref{new1}. Consequently, we have $S_0 \leq 2^{n+2}$.
We now consider $S_1$, and write
\allowdisplaybreaks
\begin{align*}
 S_1 & = \sum_{\substack{\beta \in \F_{2^n} \\ \Trnn((1+c)\beta)\neq 0}} (-1)^{\displaystyle{\Tr(\beta(F(a)+b+c\delta^{2^{2m-2}+2^{m-2}+1}))}} \\
& \qquad \qquad \sum_{X \in \F_{2^n}} (-1)^{\displaystyle{\Tr(u_1X^{2^{m-1}+1}+u_2 X^{2^{1}+1} +u_3 X^{2^{m-2}+1}+u_4 X^{2^{2}+1}+wX)}}\\
 & = \sum_{\substack{\beta \in \F_{2^n} \\ \Trnn((1+c)\beta)\neq 0}} (-1)^{\displaystyle{\Tr(\beta(F(a)+b+c\delta^{2^{2m-2}+2^{m-2}+1}))}}
\sum_{X \in \F_{2^n}} \mathcal{W}_G(w),
\end{align*}
where $ \mathcal{W}_G(w)$ is the Walsh transform of the trace of function $G(X) = u_1X^{2^{m-1}+1}+u_2 X^{2^{1}+1} +u_3 X^{2^{m-2}+1}+u_4 X^{2^{2}+1}$ at $w$. From Lemma~\ref{walsh}, the square of the Walsh transform coefficient of $G$ is given by
 \begin{equation*} 
 \mathcal{W}_G(w) ^2 =
  \begin{cases}
   2^{n+\ell} &~\mbox{if}~G(X)+\Tr(wX)\equiv0~\mbox{on Ker}~(L)  \\
    0 &~\mbox{otherwise},
  \end{cases}
 \end{equation*}
 where $\ell$ is dimension of the kernel of the linearized polynomial 
 $$
 L(X)=u_1(X^{2^m}+X)^{2^{m-1}}+u_2(X^{2^m}+X)^2+u_3(X^{2^m}+X)^{2^{m-2}}+u_4(X^{2^m}+X)^{2^2}.
 $$ 
It is straightforward to see that $\F_{2^m} \subseteq$ Ker$(L) $. Therefore, if we can show that $G(X)+\Tr(wX) \neq 0$ for all $X \in \F_{2^m}$, then $S_1=0$. We shall now make efforts to prove that 
$$
G(X)+\Tr(wX)=u_1X^{2^{m-1}+1}+u_2 X^{2^{1}+1} +u_3 X^{2^{m-2}+1}+u_4 X^{2^{2}+1}+\Tr(wX)
$$ 
is not identically zero on $\F_{2^m}$. Observe that for $X \in \F_{2^m}$, $G(X) +\Tr(wX)$ reduces to the following polynomial over $\F_{2^m}$,
$$
u_1X^{2^{m-1}+1}+u_2 X^{2^{1}+1} +u_3 X^{2^{m-2}+1}+u_4 X^{2^{2}+1}+\Trnn(\beta(1+c))X+\cdots+(\Trnn(\beta(1+c))X)^{2^{m-1}}.
$$
If $m=1$ then  $G(X) +\Tr(wX)=0$ implies that $\Trnn(\beta(1+c))=0$, which is not possible. 
For $m \geq 2$, the degree of $G(X) +\Tr(wX)$ is strictly less than that of $2^m-1$. Hence, the claim is shown.
\end{proof}

We now give an example depicting the above Theorem~\ref{T1}.
\begin{exmp}
 Consider $F(X)=(X^{8}+X+\delta)^{19}+X$, a permutation polynomial over $\F_{2^{6}}$, where $\F_{2^{6}}^{*}=\langle g \rangle $ with $g$ as a primitive element of $\F_{2^{6}}$ and $\delta=g^{54} \in \F_{2^3}$. Then $F(X)=(X^{8}+X+g^{54})^{19}+X$ is P$c$N for all $c \in \{0,g^9,g^{18},g^{27},g^{36},g^{45},g^{54}\}=S$ (say) and AP$c$N for all $c \in \F_{2^6} \setminus \{1, S\}$. This illustrates the first assertion in Theorem~\textup{\ref{T1}}.
 
 However for the second assertion, let $\delta=g^{43} \in \F_{2^6} \setminus \F_{2^3}$, where $\Tr_{1}^{6} (g^{43})=1$ and $p_m((\delta + \bar \delta)^{-1}) \neq 0$. Our experiments reveal that $F(X)=(X^{8}+X+g^{43})^{19}+X$ is P$c$N for all $c$ in $S$. Otherwise, the $c$-differential uniformity of $F$ is at most four for all $c \in \F_{2^6} \setminus \{1,S\}$. For example, when $c=g^{20}$, $F$ has $c$-differentially uniformity four.
\end{exmp}

Next, we discuss the $c$-differential uniformity of $F(X) = (X^{2^m}+X+\delta)^{3\cdot 2^{2m-2}+2^{m-2}}+X$ for some fixed values of $c$ and $\delta$.
\begin{thm} \label{T3}
Let $F(X)=(X^{2^m}+X+\delta)^{3 \cdot 2^{2m-2}+2^{m-2}}+X$ over $\F_{2^{n}}$, where $n=2m$ and $m \not\equiv 0 \pmod 3$. Then $F$ is P$c$N for all $c \in \F_{2^m} \setminus \{1\}$ and $\delta \in \F_{2^n}$.
\end{thm}
\begin{proof}
We know $F(X)$ is a permutation polynomial from Lemma~\ref{L05}. One can easily simplify $F(X)$ to get the below expression
\allowdisplaybreaks\begin{align*}
F(X) & = \Trnn(\delta)^{2^{m-2}} \Trnn(X^{2^{m-1}+2^{m-2}} +X^{2^{m-1}+2^{2m-2}}) +  \delta^{3 \cdot 2^{2m-2}+2^{m-2}}\\
&  + (\delta^{{2^{m-2}}+2^{2m-2}}+\delta^{2^{2m-1}}) \Trnn(X)^{2^{m-1}} + X^{2^m} + \delta^{2^{2m-1}}\Trnn(\delta)^{2^{m-2}} \Trnn(X^{2^{m-2}}).
\end{align*}
Notice that for $\delta \in \F_{2^m}$, $F(X)= X^{2^m} + \delta^{2^{m}}$, would have $c$-differential uniformity $1$. Let us assume $\Trnn(\delta) \neq 0$. One may recall that for any $(a, b) \in \F_{2^{n}} \times \F_{2^{n}}$,
the $c$-DDT entry $_c\Delta_F(a, b)$ is given by the number of solutions $X \in \F_{2^{n}}$ of the  equation $F(X+a)+cF(X)=b$, or equivalently,
\allowdisplaybreaks\begin{align*}
 & (1+c)F(X)+\Trnn(\delta^{2^{m-2}}) \Trnn((a^{2^{m-2}}+a^{2^{2m-2}})X^{2^{m-1}}+a^{2^{m-1}}(X^{2^{m-2}}+X^{2^{2m-2}})) \\
 & \qquad \qquad \qquad + F(a)+ \delta^{3 \cdot 2^{2m-2}+2^{m-2}} = b.
\end{align*}
 Now, by using Equation~\eqref{ddtw}, the number of solutions $X \in \F_{2^n}$ of the above equation, $_c\Delta_F(a,b)$, is given by 
\allowdisplaybreaks\begin{align*}
& 2^n \  _c\Delta_F(a,b) = \sum_{\beta \in \F_{2^n}} \sum_{X \in \F_{2^n}} (-1)^{\displaystyle{\Tr(\beta((1+c)F(X)+  F(a)+ \delta^{3 \cdot 2^{2m-2}+2^{m-2}}+b))}} \\
 & \quad (-1)^{\displaystyle{\Tr(\beta\Trnn(\delta^{2^{m-2}}) \Trnn((a^{2^{m-2}}+a^{2^{2m-2}})X^{2^{m-1}}+a^{2^{m-1}}(X^{2^{m-2}}+X^{2^{2m-2}})))}}\\
 & = \sum_{\beta \in \F_{2^n}} (-1)^{\displaystyle{\Tr(\beta(  F(a)+ \delta^{3 \cdot 2^{2m-2}+2^{m-2}}+b))}}  \sum_{X \in \F_{2^n}} (-1)^{\displaystyle{\Tr(\beta((1+c)F(X)))}} \\
 & (-1)^{\displaystyle{\Tr(\beta \Trnn(\delta^{2^{m-2}}) \Trnn((a^{2^{m-2}}+a^{2^{2m-2}})X^{2^{m-1}}+a^{2^{m-1}}(X^{2^{m-2}}+X^{2^{2m-2}})))}}.
\end{align*}
Using
\allowdisplaybreaks\begin{align*}
 T_0 & = \Tr(\beta(1+c)F(X)), \\
 T_1 & = \Tr(\beta\Trnn(\delta^{2^{m-2}}) \Trnn((a^{2^{m-2}}+a^{2^{2m-2}})X^{2^{m-1}}+a^{2^{m-1}}(X^{2^{m-2}}+X^{2^{2m-2}}))),
\end{align*}
  the above equation becomes
\begin{equation}
\label{eq222}
 _c\Delta_F(a,b) = \frac{1}{2^n} \sum_{\beta \in \F_{2^n}} (-1)^{\displaystyle{\Tr\left(\beta\left(F(a)+b+c\delta^{3 \cdot 2^{2m-2}+2^{m-2}}\right)\right)}}\sum_{X \in \F_{2^n}} (-1)^{\displaystyle{T_0+T_1}}.
\end{equation}
Further, we simplify $T_0$ and $T_1$ as follows:
\allowdisplaybreaks
\begin{align*}
T_1 & = \Tr(\beta \Trnn(\delta^{2^{m-2}}) \Trnn((a^{2^{m-2}}+a^{2^{2m-2}})X^{2^{m-1}}+a^{2^{m-1}}(X^{2^{m-2}}+X^{2^{2m-2}}))) \\
& = \Tr( \Trnn(\delta^{2^{m-1}}) (a^{2^{m-1}}+a^{2^{2m-1}})\Trnn(\beta^2) X+ \Trnn(\delta) \Trnn(a^{2})\Trnn(\beta^4) X),
\end{align*}
and 
\allowdisplaybreaks
\begin{align*}
T_0 & = \Tr(\beta(1+c)F(X)) \\
& = \Tr\left(\beta(1+c)( \Trnn(\delta)^{2^{m-2}} \Trnn(X^{2^{m-1}+2^{m-2}} +X^{2^{m-1}+2^{2m-2}}))\right)\\
& \qquad + \Tr( \beta(1+c)\delta^{2^{2m-1}}\Trnn(\delta)^{2^{m-2}} \Trnn(X^{2^{m-2}}))\\
& \qquad + \Tr\left(\beta(1+c) ((\delta^{{2^{m-2}}+2^{2m-2}}+\delta^{2^{2m-1}}) \Trnn(X)^{2^{m-1}} + X^{2^m} + \delta^{3 \cdot 2^{2m-2}+2^{m-2}}) \right)\\
& = \Tr(\beta(1+c)\delta^{3 \cdot 2^{2m-2}+2^{m-2}}) + \Tr((\beta(1+c))^{2^m} X+\Trnn(\delta)\Trnn(\delta^{2} (\beta(1+c))^4) X) \\
&\qquad \quad + \Tr\left(\Trnn((\beta(1+c))^2(\delta + \delta^{2^{m-1}+2^{2m-1}}))X\right)\\
&\qquad \qquad  +\Trnn(\delta)\Trnn((\beta(1+c))^4)(X^{3}+X^{2^{m+1}+1}).
\end{align*}

Now, Equation~\eqref{eq222} reduces to
\allowdisplaybreaks
\begin{align*}
 _c\Delta_F(a,b) &= \frac{1}{2^n} \sum_{\beta \in \F_{2^n}} (-1)^{\displaystyle{\Tr\left(\beta\left(F(a)+b+c\delta^{3 \cdot 2^{2m-2}+2^{m-2}}\right)\right)}}\\
 & \qquad \quad \sum_{X \in \F_{2^n}} (-1)^{\displaystyle{\Tr(u(X^{2^{1}+1}+X^{2^{m+1}+1})+wX)}},
\end{align*}
where 
\allowdisplaybreaks
\begin{align*}
 u & = \Trnn(\delta)\Trnn((\beta(1+c))^4)=(1+c)^4 \Trnn(\delta)\Trnn(\beta)^4 \\
 w & =  \Trnn(\delta^{2^{m-1}}) \Trnn(a^{2^{m-1}})\Trnn(\beta^2) + \Trnn(\delta) \Trnn(a^{2})\Trnn(\beta^4)\\
 & \quad +(\beta(1+c))^{2^m}+ \Trnn(\delta)\Trnn(\delta^{2} (\beta(1+c))^4 )+ \Trnn((\beta(1+c))^2(\delta + \delta^{2^{m-1}+2^{2m-1}})) \\
 & =  \Trnn(\delta^{2^{m-1}}) \Trnn(a^{2^{m-1}})\Trnn(\beta^2) + \Trnn(\delta) \Trnn(a^{2})\Trnn(\beta^4)\\
 & \quad +(1+c)\beta^{2^m}+ (1+c)^4\Trnn(\delta)\Trnn(\delta^{2} \beta^4) +(1+c)^2 \Trnn(\beta^2(\delta + \delta^{2^{m-1}+2^{2m-1}})).
\end{align*}
Further, splitting the above sum depending on whether $\Trnn(\beta)$ is~$0$  or not, we get
\allowdisplaybreaks
\begin{align*}
 2^n\ _c\Delta_F(a,b)= & \sum_{\substack{\beta \in \F_{2^n} \\ \Trnn(\beta)=0}} (-1)^{\displaystyle{\Tr(\beta(F(a)+b+c\delta^{3 \cdot 2^{2m-2}+2^{m-2}}))}} \sum_{X \in \F_{2^n}} (-1)^{\displaystyle{\Tr((1+c)\beta^{2^m}X)}} \\
 & (-1)^{\displaystyle{\Tr(( (1+c)^4\Trnn(\delta)\Trnn(\delta^{2} \beta^4) +(1+c)^2 \Trnn(\beta^2(\delta + \delta^{2^{m-1}+2^{2m-1}})))X)}}\\
 & +  \sum_{\substack{\beta \in \F_{2^n} \\ \Trnn(\beta)\neq 0}} (-1)^{\displaystyle{\Tr(\beta(F(a)+b+c\delta^{3\cdot 2^{2m-2}+2^{m-2}}))}}\\
 & \qquad  \sum_{X \in \F_{2^n}} (-1)^{\displaystyle{\Tr(u(X^{2^{1}+1}+X^{2^{m+1}+1})+wX)}}\\
 &= S_0 + S_1,
\end{align*}
where $S_0,S_1$ are the two inner sums. We first consider the sum $S_0$ below. We write
\allowdisplaybreaks
\begin{align*}
 S_0 & = \sum_{\substack{\beta \in \F_{2^n} \\ \Trnn(\beta)=0}} (-1)^{\displaystyle{\Tr(\beta(F(a)+b+c\delta^{3\cdot 2^{2m-2}+2^{m-2}}))}} \sum_{X \in \F_{2^n}} (-1)^{\displaystyle{\Tr\left((1+c)\beta^{2^m}X\right)}} \\
 & \qquad (-1)^{\displaystyle{\Tr\left(( (1+c)^4\Trnn(\delta)\Trnn(\delta^{2} \beta^4) +(1+c)^2 \Trnn(\beta^2(\delta + \delta^{2^{m-1}+2^{2m-1}})))X \right)}}\\
 & = \sum_{\substack{\beta \in \F_{2^n} \\ \Trnn(\beta)=0}} (-1)^{\displaystyle{\Tr(\beta(F(a)+b+c\delta^{3\cdot 2^{2m-2}+2^{m-2}}))}} \sum_{X \in \F_{2^n}} (-1)^{\displaystyle{\Tr\left(\beta (1+c) X\right)}} \\
 & \qquad (-1)^{\displaystyle{\Tr\left(( (1+c)^4\Trnn(\delta)^3 \beta^4 +(1+c)^2 \Trnn(\delta)\beta^2)X \right)}}
  = 2^n. 
\end{align*}
The last identity follows by analyzing the number of solutions $\beta$ of the following equation,
\begin{equation}\label{eqqq}
 \beta(1+c) +  (1+c)^4\Trnn(\delta)^3 \beta^4 +(1+c)^2 \Trnn(\delta) \beta^2 =0,
\end{equation}
or equivalently, nonzero solutions $\beta \in \F_{2^m}$ of $Z^3+Z+1=0$, where $Z= (1+c)\Trnn(\delta) \beta$. From Lemma~\ref{L01}, it is clear that the polynomial $p_m(X)$ has an odd number of terms if $m \not \equiv 0 \pmod 3$, and each term is a monomial of $X$. Hence, it cannot have three distinct solutions in $\F_{2^m}$. Also, since $\Trm(1+1)=0$, it cannot have a unique solution. Thus, Equation~\eqref{eqqq} has only one solution $\beta=0$ in $\F_{2^m}$, and that gives us $S_0 = 2^n$. Next, we have
\allowdisplaybreaks
\begin{align*}
 S_1 & = \sum_{\substack{\beta \in \F_{2^n} \\ \Trnn(\beta)\neq 0}} (-1)^{\displaystyle{\Tr(\beta(F(a)+b+c\delta^{3\cdot 2^{2m-2}+2^{m-2}}))}} \\
& \qquad \qquad \sum_{X \in \F_{2^n}} (-1)^{\displaystyle{\Tr\left(u(X^{2^{1}+1}+X^{2^{m+1}+1})+wX\right)}} \\
 & = \sum_{\substack{\beta \in \F_{2^n} \\ \Trnn(\beta)\neq 0}} (-1)^{\displaystyle{\Tr(\beta(F(a)+b+c\delta^{3 \cdot 2^{2m-2}+2^{m-2}}))}}
\sum_{X \in \F_{2^n}} \mathcal{W}_{f_u}(w),
\end{align*}
where  $\mathcal{W}_{f_u}(w)$ is the Walsh transform of the function $f_u (X)=  \Tr( u(X^{2^{1}+1}+X^{2^{m+1}+1}))$ at $w$.  
Next, we divide our proof into three cases based on the value of $m$. Also, we have $k=2m$, $a =1$ and $b=m+1$.

\textbf{Case 1.} Let $m$ be odd. Then $\gcd(b-a,k)=m$ and $\gcd(b+a,k)=1$. Also, notice that $v_2(a)= v_2(1)=0$, $v_2(b) = v_2(m+1)$ and $v_2(k)=v_2(2m)=1$. Summarizing, we observe that $v_2(a) = v_2(b) \leq v_2(k)$ does not hold. From Lemma~\ref{L011}, one can see that $\mathcal{W}_{f_{u}}(w)= \mathcal{W}_f(\frac{w}{\eta})$, where $f(X) = \Tr( X^{2^{1}+1}+X^{2^{m+1}+1})$ and $\eta \in \F_{2^m}$ such that $u = \eta^{2^{1}+1}$. We now show that $\mathcal{W}_f(\frac{w}{\eta})=0$ via Lemma~\ref{L010}. Notice that $0=v_2(b-a)=v_2(b+a)=v_2(k)-1$.
Hence it is sufficient to show that $(L_1 \circ L_2)(\frac{w}{\eta}+1)\neq 0$. Now,
$$
(L_1 \circ L_2)(X) = X+ X^2 +X^{2^2} + \cdots + X^{2^{m-1}}.
$$
Hence, 
\allowdisplaybreaks\begin{align*}
 &(L_1 \circ L_2)\left(\frac{w}{\eta}+1\right)  =m+ \frac{\beta^{2^m}(1+c)}{\eta} + \left(\frac{\beta^{2^m}(1+c)}{\eta}\right)^2 +\cdots +\left(\frac{\beta^{2^m}(1+c)}{\eta}\right)^{2^{m-1}}\\
 &\qquad + \Trm\left( \frac{ \Trnn(\delta^{2^{m-1}}) \Trnn(a^{2^{m-1}})\Trnn(\beta^2) + \Trnn(\delta) \Trnn(a^{2})\Trnn(\beta^4)}{\eta}\right) \\
 &\qquad +\Trm\left( \frac{ (1+c)^4\Trnn(\delta)\Trnn(\delta^{2^{2m+1}} \beta^4) +(1+c)^2 \Trnn(\beta^2(\delta + \delta^{2^{m-1}+2^{2m-1}}))}{\eta}\right).\\
\end{align*}
Now, if $(L_1 \circ L_2)\left(\frac{w}{\eta}+1\right)=0$, then we have $\left((L_1 \circ L_2)\left(\frac{w}{\eta}+1\right)\right)^2=0$. This will give us 
\allowdisplaybreaks 
\begin{equation}
\begin{split}
  (L_1 \circ L_2)\left(\frac{w}{\eta}+1\right)+ \left((L_1 \circ L_2)\left(\frac{w}{\eta}+1\right)\right)^2 & =\frac{1}{\eta} \Trnn(\beta^{2^m}(1+c)) +m +m^2\\
  & =\frac{(1+c)}{\eta} \Trnn(\beta) + m+m^2 = 0.
  \end{split}
\end{equation}
As $u = \eta^3$, we have $\eta \neq 0$, that gives us $\Trnn(\beta) = 0$ which is obviously not true and hence we are done.

\textbf{Case 2.} Let $m \equiv 0 \pmod 4$. If we have $v_2(m)=v ( \geq 2)$, then $v_2(a)= v_2(b) = 0$ and $v_2(k)=v_2(2m)=v+1$. Summarizing, we have $v_2(a)= v_2(b) < v_2(k)$ and $v_2(k) > \nu$. Thus, by using Lemma~\ref{L011} if we show that $L_1 \circ L_2 \left(\frac{w}{u^2}\right) \neq 0$ then we have $\mathcal{W}_G(w)=0$ which gives us $S_1 =0$ and then we are done with the claim. 
So, our goal is to now show that $L_1 \circ L_2 \left(\frac{w}{u^2}\right)  \neq 0$. As $d_1 =m$ and $d_2=2$, we have 
$$
(L_1 \circ L_2)(X) = X+ X^{2^2} +X^{2^{(2\cdot2)}} + X^{2^{(3\cdot2)}} +\cdots + X^{2^{m-2}}.
$$ 
Now, if $(L_1 \circ L_2)\left(\frac{w}{u^2}\right)=0$, then we have $\left((L_1 \circ L_2)\left(\frac{w}{u^2}\right)\right)^{2^2}=0$. This will give us 
\begin{equation*}
  (L_1 \circ L_2)\left(\frac{w}{u^2}\right) + \left((L_1 \circ L_2)\left(\frac{w}{u^2}\right)\right)^{2^2} = \Trnn\left(\frac{w}{u^2}\right)= \frac{1}{u^2} \Trnn(w)  =0,
\end{equation*}
which is a contradiction to the assumption that $\Trnn(\beta)\neq0$.

\textbf{Case 3.} Let $m \equiv 2 \pmod 4$, then $d_1=m$ and $d_2= 4$. In this subcase, we have $v_2(a)= v_2(b) = 0$ and $v_2(k)=v_2(2m)=2$. Summarizing, we observe that $v_2(a) = v_2(b) < v_2(k)$ holds and $v_2(k) \leq \nu$. Then again from Lemma~\ref{L011}, if we show that $ \frac{w}{u^2} \not \in S_{d_1} \cap S_{d_2}$, then we are done. In this subcase, $S_{d_1} = \{X \in \F_{2^n} : \Trnn(X) = 0\}$ and $ S_{d_2} = \{X \in \F_{2^n} : \Trfn(X) = 0\}$.
Let $ \frac{w}{u^2} \in S_{d_1} \cap S_{d_2}$. Then we have $\Trnn\left( \frac{w}{u^2}\right) = 0$, which implies that $\Trnn(\beta) = 0$, which is not possible. Hence $\mathcal{W}_G(w)=0$, giving us $S_1 =0$. This shows the claim that $F$ is P$c$N for $c \in \F_{2^m}\setminus \{1\}$, and the proof is done.
\end{proof}

The following is an example illustrating Theorem~\ref{T3}.
\begin{exmp}
Let $F(X)=(X^{16}+X+\delta)^{196}+X$, a permutation polynomial over $\F_{2^8}$, where $\F_{2^{8}}^{*}=\langle g \rangle $ and $g$ is a primitive element of $\F_{2^{8}}$. Then, $F(X)$ is P$c$N for all $\delta \in \F_{2^8}$ and $c \in \{0, g^{17}, g^{34}, g^{51}, g^{68}, g^{85}, g^{102}, g^{119}, g^{136}, g^{153}, g^{170}, g^{187}, g^{204}, g^{221}, g^{238}, g^{255}\}$.
\end{exmp}

\begin{thm}\label{T4}
 Let $F(X)=(X^{2^m}+X+\delta)^{3 \cdot 2^{m-2}+2^{2m-2}}+X$ over $\F_{2^{n}}$, where $n=2m$ and $m \not\equiv 0 \pmod 3$. Then $F$ is P$c$N for all $c \in \F_{2^m} \setminus \{1\}$ and $\delta \in \F_{2^n}$. 
\end{thm}
\begin{proof}
 Notice that $F(X)$ is a permutation polynomial over $\F_{2^n}$ from Lemma~\ref{L06}. The proof follows along a similar lines as in the aforementioned Theorem~\ref{T3}.
\end{proof}

We now present the following example illustrating Theorem~\ref{T4}
\begin{exmp}
Consider $F(X)=(X^{32}+X+\delta)^{280}+X$, a permutation polynomial over $\F_{2^{10}}$, where $\F_{2^{10}}^{*}=\langle g \rangle $ and $g$ is a primitive element of $\F_{2^{10}}$. Then, $F(X)$ is P$c$N for all $\delta \in \F_{2^{10}}$ and $c \in \{0, g^{33k} ~\mid ~1 \leq k \leq 31\}$.
\end{exmp}

Next, we compute the $c$-differential uniformity of the permutation polynomial $F(X)=(X^{2^m}+X+\delta)^{2^{2m+1}+2^m}+(X^{2^m}+X+\delta)^{2^{2m}+2^{m+1}}+X$ and show that it is P$c$N for some values of $c$ and $\delta$. 

\begin{thm} \label{T5}
Let $F(X)=(X^{2^m}+X+\delta)^{2^{2m+1}+2^m}+(X^{2^m}+X+\delta)^{2^{2m}+2^{m+1}}+X$ over $\F_{2^{n}}$, where $n=3m$. Let $\delta\in\F_{2^n}$ and $\Trmn(\delta)=0$. Then $F$ is P$c$N for $c \in \F_{2^m}\setminus \{1\}$.
\end{thm}
\begin{proof}
First, we know that $F$ is a permutation polynomial via Lemma~\ref{L03}. 
 Next, by expanding the trinomial, one can easily simplify $F(X)$ and gets the below expression
\allowdisplaybreaks\begin{align*}
F(X) & = \Trmn(X^{2^{m+1}+1} +X^{2^{2m+1}+1}) + (\delta^{2^{m+1}} + \delta^{2^{2m+1}}) X^{2^{2m}} + \delta^{2^{2m+1}} X^{2^{m}} + \delta^{2^{2m}} X^{2^{m+1}}\\
&  + (\delta^{2^m}+\delta^{2^{2m}}) X^{2^{2m+1}} + \delta^{2^m} X^2 + (\delta^{2^{m+1}}+1)X + \delta^{2^{m+1}+2^{2m}} +  \delta^{2^{2m+1}+2^{m}}.
\end{align*}
Since $\Trmn(\delta)=0$, we can write the above equation as
\allowdisplaybreaks
\begin{align*}
F(X) & = \Trmn(X^{2^{m+1}+1} +X^{2^{2m+1}+1}) + \delta^2 X^{2^{2m}} + \delta^{2^{2m+1}} X^{2^{m}} + \delta^{2^{2m}} X^{2^{m+1}}\\
&  + \delta X^{2^{2m+1}} + \delta^{2^m} X^2 + (\delta^{2^{m+1}}+1)X + \delta^{2^{m+1}+2^{2m}} +  \delta^{2^{2m+1}+2^{m}}.
\end{align*}
Recall that, for any $(a, b) \in \F_{2^{n}} \times \F_{2^{n}}$,
the $c$-DDT entry $_c\Delta_F(a, b)$ is given by the number of solutions $X \in \F_{2^{n}}$ of the equation $ F(X+a)+cF(X)=b$, or equivalently,
\allowdisplaybreaks
\begin{align*}
 & (1+c)F(X)+F(a)+ \delta^{2^{m+1}+2^{2m}} +  \delta^{2^{2m+1}+2^{m}}\\
 & \qquad \qquad \qquad + \Trmn((a^{2^{m+1}} +a^{2^{2m+1}})X +a(X^{2^{m+1}} +X^{2^{2m+1}})) = b.
\end{align*}
From Equation~\eqref{ddtw}, the number of solutions $X \in \F_{2^n}$ of the above equation is given by 
\allowdisplaybreaks\begin{align*}
& 2^n \  _c\Delta_F(a,b) = \sum_{\beta \in \F_{2^n}} \sum_{X \in \F_{2^n}} (-1)^{\displaystyle{\Tr(\beta((1+c)F(X)+F(a)+ \delta^{2^{m+1}+2^{2m}} +  \delta^{2^{2m+1}+2^{m}}))}} \\
 & \qquad \qquad (-1)^{\displaystyle{\Tr(\beta(\Trmn((a^{2^{m+1}} +a^{2^{2m+1}})X +a(X^{2^{m+1}} +X^{2^{2m+1}})) + b))}},
\end{align*}
or equivalently,
\begin{equation*}
 \begin{split}
 &2^n \ _c\Delta_F(a,b)  =  \sum_{\beta \in \F_{2^n}} (-1)^{\displaystyle{\Tr\left(\beta\left(F(a)+b+\delta^{2^{m+1}+2^{2m}} +  \delta^{2^{2m+1}+2^{m}}\right)\right)}}  \\
  &\qquad \sum_{X \in \F_{2^n}} (-1)^{\displaystyle{\Tr\left(\beta(1+c)F(X) + \beta \Trmn((a^{2^{m+1}} +a^{2^{2m+1}})X +a(X^{2^{m+1}} +X^{2^{2m+1}})) \right)}}.
 \end{split}
\end{equation*}
Using notations
\allowdisplaybreaks\begin{align*}
 T_0 & = \Tr(\beta(1+c)F(X)) \\
 T_1 & = \Tr(\beta( \Trmn((a^{2^{m+1}} +a^{2^{2m+1}})X +a(X^{2^{m+1}} +X^{2^{2m+1}})))),
\end{align*}
 the above equation becomes
\begin{equation}\label{eq331}
 _c\Delta_F(a,b) = \frac{1}{2^n} \sum_{\beta \in \F_{2^n}} (-1)^{\displaystyle{\Tr\left(\beta\left(F(a)+b+\delta^{2^{m+1}+2^{2m}} +  \delta^{2^{2m+1}+2^{m}}\right)\right)}}\sum_{X \in \F_{2^n}} (-1)^{\displaystyle{T_0+T_1}}.
\end{equation}
Let $c\in \F_{2^m} \setminus \{1\}$. To compute $T_0$ and $T_1$, we simplify them further,
\allowdisplaybreaks
\begin{equation*} 
\begin{split}
T_1 & = \Tr(\beta( \Trmn((a^{2^{m+1}} +a^{2^{2m+1}})X +a(X^{2^{m+1}} +X^{2^{2m+1}})))) \\
& = (a^{2^{m+1}} +a^{2^{2m+1}}) \Trmn(\beta) X +  (a^{2^{m-1}} +a^{2^{2m-1}}) \Trmn(\beta^{2^{m-1}}) X,
\end{split}
\end{equation*}
and 
\allowdisplaybreaks 
\begin{equation*} 
\begin{split}
T_0  & = \Tr(\beta(1+c)F(X)) \\
& = \Tr\left(\beta(1+c)(\Trmn(X^{2^{m+1}+1} +X^{2^{2m+1}+1}) + \delta^2 X^{2^{2m}} + \delta^{2^{2m+1}} X^{2^{m}}+ \delta^{2^{2m}} X^{2^{m+1}}) \right)\\
& \qquad + \Tr\left(\beta(1+c) (\delta X^{2^{2m+1}} + \delta^{2^m} X^2 + (\delta^{2^{m+1}}+1)X + \delta^{2^{m+1}+2^{2m}} +  \delta^{2^{2m+1}+2^{m}}) \right)\\
& = \Tr\left( \Trmn(\beta(1+c))(X^{2^{m+1}+1} +X^{2^{2m+1}+1})+ \delta^{2^{m+1}}((\beta(1+c))^{2^m}+(\beta(1+c))^{2^{2m}})X\right) \\
& \qquad +\Tr\left((\beta(1+c))^{2^{2m-1}} \delta^{2^{m-1}} X+ (\beta(1+c)\delta)^{2^{m-1}} X+ (\beta(1+c)\delta^{2^m})^{2^{3m-1}} X\right)\\
&  \qquad   + \Tr\left(\beta(1+c)(\delta^{2^{m+1}}+1)X + \beta(1+c)( \delta^{2^{m+1}+2^{2m}} +  \delta^{2^{2m+1}+2^{m}})\right)\\
& = \Tr\left( \Trmn(\beta(1+c))(X^{2^{m+1}+1} +X^{2^{2m+1}+1})+ \beta(1+c) X+ \Trmn (\beta(1+c))\delta^{2^{m+1}} X\right) \\
& \qquad + \Tr\left( \Trmn(\beta(1+c))^{2^{m-1}} \delta^{2^{m-1}} X +\beta(1+c)( \delta^{2^{m+1}+2^{2m}} +  \delta^{2^{2m+1}+2^{m}}) \right).
\end{split}
\end{equation*}

Now Equation~\eqref{eq331} reduces to
\allowdisplaybreaks
\begin{equation*}
\begin{split}
 _c\Delta_F(a,b) & = \frac{1}{2^n} \sum_{\beta \in \F_{2^n}} (-1)^{\displaystyle{\Tr\left(\beta(F(a)+b+c( \delta^{2^{m+1}+2^{2m}} +  \delta^{2^{2m+1}+2^{m}}))\right)}} \\
 & \qquad \qquad \sum_{X \in \F_{2^n}} (-1)^{\displaystyle{\Tr(u(X^{2^{m+1}+1} +X^{2^{2m+1}+1})+wX)}},
 \end{split}
\end{equation*}
where 
\allowdisplaybreaks
\begin{align*}
 u & = \Trmn((1+c)\beta) \\
 w & =  (a^{2^{m+1}} +a^{2^{2m+1}}) \Trmn(\beta) +  (a^{2^{m-1}} +a^{2^{2m-1}}) \Trmn(\beta^{2^{m-1}}) + \beta(1+c) \\
 & \qquad +\Trmn (\beta(1+c))\delta^{2^{m+1}}+ \Trmn(\beta(1+c))^{2^{m-1}}\delta^{2^{m-1}}.
\end{align*}

Since $c \in \F_{2^m} \setminus \{1\}$, we can split the above sum depending on whether $\Trmn((1+c)\beta)$ is~$0$ or not, or equivalently,  $\Trmn(\beta)$ is~$0$ or not. We write
\allowdisplaybreaks
\begin{align*}
 2^n \ _c\Delta_F(a,b) & = \sum_{\substack{\beta \in \F_{2^n} \\ \Trmn(\beta)=0}} (-1)^{\displaystyle{\Tr(\beta(F(a)+b+c( \delta^{2^{m+1}+2^{2m}} +  \delta^{2^{2m+1}+2^{m}})))}}\\
 & \qquad \sum_{X \in \F_{2^n}} (-1)^{\displaystyle{\Tr\left(\left(\beta(1+c)+(a^{2^{m+1}} +a^{2^{2m+1}}) \Trmn(\beta) \right)X\right)}} \\
 & \qquad \qquad \qquad \qquad (-1)^{\displaystyle{\Tr\left((a^{2^{m-1}} +a^{2^{2m-1}}) \Trmn(\beta^{2^{m-1}})X \right)}}\\
 & +  \sum_{\substack{\beta \in \F_{2^n} \\ \Trmn(\beta)\neq 0}} (-1)^{\displaystyle{\Tr(\beta(F(a)+b+c( \delta^{2^{m+1}+2^{2m}} +  \delta^{2^{2m+1}+2^{m}})))}}\\
 & \qquad \qquad \qquad \sum_{X \in \F_{2^n}} (-1)^{\displaystyle{\Tr(u(X^{2^{m+1}+1} +X^{2^{2m+1}+1})+wX)}}\\
 &= S_0 + S_1,
\end{align*}
where $S_0,S_1$ are the two inner sums. We first consider the first sum
$$
S_0=\sum_{\substack{\beta \in \F_{2^n} \\ \Trmn(\beta)=0}} (-1)^{\displaystyle{\Tr\left(\beta\left(F(a)+b+c( \delta^{2^{m+1}+2^{2m}} +  \delta^{2^{2m+1}+2^{m}})\right)\right)}}\sum_{X \in \F_{2^n}} (-1)^{\displaystyle{\Tr\left(\beta(1+c)X\right)}}.
$$
Clearly, the inner sum in $S_0$ is zero only for $\beta=0$ and hence we have $S_0=2^n$. Next, we consider the second sum
\allowdisplaybreaks
\begin{align*}
 S_1 & = \sum_{\substack{\beta \in \F_{2^n} \\ \Trmn(\beta)\neq 0}} (-1)^{\displaystyle{\Tr\left(\beta\left(F(a)+b+c( \delta^{2^{m+1}+2^{2m}} +  \delta^{2^{2m+1}+2^{m}})\right)\right)}}
\sum_{X \in \F_{2^n}} \mathcal{W}_{f_u}(w),
\end{align*}
where  $\mathcal{W}_{f_u}(w)$ is the Walsh transform of the function $f_u (X)=   \Tr(u(X^{2^{m+1}+1} +X^{2^{2m+1}+1}))$ at $w$.  
We split our analysis into three cases depending on the value of $m$. Note that $k=3m$, $a =m+1$ and $b=2m+1$.

\textbf{Case 1.} Let $m$ be odd. Then $\gcd(b-a,k)=m$ and $\gcd(b+a,k)=1$. Also, notice that $v_2(a)= v_2(m+1)$, $v_2(b) = v_2(2m+1) = 0$ and $v_2(k)=v_2(3m)=0$. Summarizing all, we observe that $v_2(a) = v_2(b) \leq v_2(k)$ does not hold. From Lemma~\ref{L011}, one can see that $\mathcal{W}_{f_{u}}(w)= \mathcal{W}_f(\frac{w}{\eta})$, where $f(X) =  \Tr(X^{2^{m+1}+1}+X^{2^{2m+1}+1})$ and $\eta \in \F_{2^m}$ such that $u = \eta^{2^{m+1}+1}$. Also, observe that $0=v_2(b-a)=v_2(b+a)\neq v_2(k)-1$ and $v_2(k)=\nu$. Hence it is sufficient to show that $\frac{w}{\eta} \not \in S_{d_1} \cap  S_{d_2}$, where $S_{d_1} = \{X \in \F_{2^n} : \Trmn(X) = 0\}$ and $ S_{d_2} = \{X \in \F_{2^n} : \Trn(X) = 0\}$. It is easy to observe that $\Trmn\left(\frac{w}{\eta}\right)=\frac{1}{\eta}\Trmn(\beta(1+c))=0$, which is not possible as $\Trmn(\beta(1+c)) \neq 0$. Hence $\mathcal{W}_{f_{u}}(w)= \mathcal{W}_f(\frac{w}{\eta})=0$.

\textbf{Case 2.} Let $m \equiv 0 \pmod 4$. If we have $v_2(m)=v ( \geq 2)$, then $v_2(a)= v_2(b) = 0$ and $v_2(k)=v_2(3m)=v_2(m)$. Summarizing all, we have $v_2(a)= v_2(b) < v_2(k)$ and $v_2(k) = \nu$. Thus, by using Lemma~\ref{L011} if we show that $\frac{w}{u^{2^{-(2m+1)}}} = \frac{w}{u^{2^{m-1}}} \not \in S_{d_1} \cap S_{d_2}$, where 
$S_{d_1} = \{X \in \F_{2^n} : \Trmn(X) = 0\}$ and $ S_{d_2} = \{X \in \F_{2^n} : \Trtn(X) = 0\}$. Since $u \in \F_{2^m}$, $\Trmn\left(\frac{w}{u^{2^{m-1}}} \right) = 0$, this implies that $\Trmn(\beta(1+c))=0$, contradicting the assumption that $\Trmn(\beta(1+c))\neq 0$. Thus, $S_1=0$ in this case too.

\textbf{Case 3.} Let $m \equiv 2 \pmod 4$. Then we have $v_2(m)=1$, $v_2(a)= v_2(b) = 0$ and $v_2(k)=v_2(3m)=v_2(m)=1$. Summarizing all, we have $v_2(a)= v_2(b) < v_2(k)$ and $v_2(k) < \nu$. By applying the same argument as in Subcase 1(b), we have $S_1=0$. This completes the proof.
\end{proof}

Here, we provide an example illustrating Theorem~\ref{T5}.
\begin{exmp}
Consider $F(X)=(X^{8}+X+g^{33})^{136}+(X^{8}+X+g^{33})^{80}+X$, a permutation polynomial over $\F_{2^{9}}$, where $\F_{2^{9}}^{*}=\langle g\rangle$ and $g$ is a primitive element of $\F_{2^{9}}$. Notice that $\Tr_{3}^{9}(g^{33})=0$. Then, $F(X)$ is P$c$N for all $c \in \{0, g^{73}, g^{146}, g^{219}, g^{292}, g^{365}, g^{438}, g^{511}\}$.
\end{exmp}

In Table~\ref{Table1}, we present a list of permutation polynomials over $\F_{2^n}$ with low $c$-differential uniformity.

 \begin{table}[h]
 \caption{A list of permutation polynomials over $\F_{2^n}$ with low $c$-differential uniformity}
 \label{Table1} 
 \begin{center}
\begin{adjustbox}{width=.8\textwidth}
 \begin{tabular}{|p{6cm}|p{9cm}|p{1cm}|p{2cm}|} 
  \hline
   $F(X)$ & Conditions & $_c\Delta_F$  & Reference \\
  \hline 
  $X^{2^n-2} + X^{2^n-1} + (X + 1)^{2^n-1}$ & $n \geq 2$, $c \in \F_{2^n} \setminus \{0, 1\}$ & $\leq 4$   &  \cite{PS}\\
  \hline
   $X+\Trn(\alpha X+X^{2^k+1})$ & $n \geq 3$ and $\gcd(k,n)=1$, $\alpha \in \F_{2^n}$ with $\Trn(\alpha)=1$, $c \in \F_{2^n} \setminus \{0, 1\}$ & $2$   &  \cite{HPS1}\\
      & $n \geq 3$ and $\gcd(k,n)=1$, $\alpha \in \F_{2^n}$ with $\Trn(\alpha)=1$, $c=0$ & $1$   &  \\
   & $n \geq 3$ and $\gcd(k,n)=1$, $\alpha \in \F_{2^n}$ with $\Trn(\alpha)=1$, $c =1$ & $2^n$   &  \\
  \hline
  $X^{-1}+\Trn\left(\dfrac{X^2}{X+1}\right)$ & all $n,c \in \F_{2^n} \setminus \{0, 1\}, \Trn(c)=\Trn(c^{-1})=1$ & $\leq 8$   & \cite{HPS1} \\
    & all $n,c \in \F_{2^n} \setminus \{0, 1\}, \Trn(c)=\Trn(c^{-1})=1$ or $\Trn(c)+\Tr(c^{-1})=1$ & $\leq 9$   & \\
   & all $n$, $c =0$ & 1 &  \\
  \hline
  $(X^{2^m}+X+\delta)^{2^{2m}+1}+X$ & $n=3m, \delta \in \F_{2^n}, c \in  \F_{2^m} \setminus \{1\}$ & $1$   & \cite{Garg} \\
    & $n=3m,  \Gamma_1=\{\delta \in \F_{2^n} \mid \Tr_{m}^{3m}(\delta)=1\}, \delta \in \Gamma_1, c \in  \F_{2^n} \setminus \F_{2^m}$ & $2$  &  \\
    & $n=3m, \delta \in \F_{2^n} \setminus \Gamma_1, c \in  \F_{2^n} \setminus \F_{2^m}$ & $\leq 4$  & \\
    \hline
    $(X^{2^m}+X+\delta)^{2^{2m-1}+2^{m-1}}+X$ & $n=3m, m \not \equiv 1 \pmod 3,  \delta \in \F_{2^n}, c \in  \F_{2^m} \setminus \{1\}$ & $1$   & \cite{Garg}  \\
    & $n=3m, m \not \equiv 1 \pmod 3, \Gamma_0=\{\delta \in \F_{2^n} \mid \Tr_{m}^{3m}(\delta)=0\}, \delta \in \Gamma_0, c \in  \F_{2^n} \setminus \F_{2^m}$ & $2$   &   \\
    & $n=3m, m \not \equiv 1 \pmod 3, \delta \in \F_{2^n} \setminus \Gamma_0, c \in  \F_{2^n} \setminus \F_{2^m}$ & $\leq 4$   & \\
    \hline
    $(X^{2^m}+X+\delta)^{2^{3m-1}+2^{m-1}}+X$ & $n=3m, 2m \not \equiv 1 \pmod 3,  \delta \in \F_{2^n}, c \in  \F_{2^m} \setminus \{1\}$ & $1$   & \cite{Garg}  \\
    & $n=3m, 2m \not \equiv 1 \pmod 3, \Gamma_0=\{\delta \in \F_{2^n} \mid \Tr_{m}^{3m}(\delta)=0\}, \delta \in \Gamma_0, c \in  \F_{2^n} \setminus \F_{2^m}$ & $2$   &   \\
    & $n=3m, 2m \not \equiv 1 \pmod 3, \delta \in \F_{2^n} \setminus \Gamma_0, c \in  \F_{2^n} \setminus \F_{2^m}$ & $\leq 4$   & \\
    \hline
\raggedright\(\begin{cases} 
\frac{1}{X}&\text{for~} X \not \in \{0,1,\gamma\}\\
0 &\text{for~} X = \gamma \\
\frac{1}{\gamma}  &\text{for~} X = 1 \\
 1  &\text{for~} X =  0 \\
\end{cases}\)  & $n \geq 4, c,\gamma \in \F_{2^n} \setminus \{0,1\}$ & $\leq 5$   & \cite{JKK2}  \\
    \hline
    $X+\Trn(X^{2^{k+1}+1}+X^3+X+uX)$ & $n=2k+1, u \in \F_{2^n}$ with $\Trn(u)=1$, $c \in  \F_{2^n} \setminus \{0,1\}$ & $2$ &  \cite{Liu}\\
    & $n=2k+1, k \not \equiv 1 \pmod 3, u \in \F_{2^n}$ with $\Trn(u)=1$, $c =1$ & $2^n$   &  \\
    \hline
    $X+\Trn(X^{2^{k+3}}+(X+1)^{2^{k+3}})$ & $n=2k+1, c \in  \F_{2^n} \setminus \{0,1\}$ & $2$ &    \cite{Liu}\\
    & $n=2k+1, k \not \equiv 1 \pmod 3,c =1$ & $2^n$   &     \\
    \hline
    $X^{-1}+\Trn(X^{-d}+(X+1)^{-d})$ & $n$ is even, any positive integer $d$, $\Trn(c)=\Trn(c^{-1})= 1, c \in  \F_{2^n} \setminus \{0,1\}$ & $\leq 8$   & \cite{Liu}\\
    & $n$ is even, any positive integer $d$, $\Trn(c)=\Trn(c^{-1})= 0$ or $\Trn(c)+\Trn(c^{-1})=1, c \in  \F_{2^n} \setminus \{0,1\}$ & $\leq 9$   & \\
    & $n$ is even, any positive integer $d$, $c=0$ & $1$   &    \\
    \hline
    $X^{-1}+\Trn(X^{-d}+(X+1)^{-d})$ & $n$ is even, any positive integer $d\in \{2^n-2,2^{\frac{n}{2}}+2^{\frac{n}{4}}+1, 2^{t_1}+1, 3(2^{t_2}+1)\}, 1 \leq t_1\leq\frac{n}{2}-1, 2 \leq t_2\leq\frac{n}{2}-1, c=1$ & $4$  &   \cite{Tan} \\
    \hline
    $(X^{2^m}+X+\delta)^{2^{2m-2}+2^{m-2}+1}+X$ 
    & $n=2m, c \in  \F_{2^m} \setminus \{1\}, \delta \in \F_{2^m}$   & $1$   &  Theorem~\ref{T1} \\
    & $n=2m, c \in  \F_{2^n} \setminus \F_{2^m}, \delta \in \F_{2^m}$   & $2$   &  \\
    & $n=2m, c \in  \F_{2^m} \setminus \{1\}, \delta \in \F_{2^n} \setminus \F_{2^m}$ with $\Tr_{1}^{2m}(\delta)=\Trm(1)$ and $p_m(\delta+\bar\delta)^{-1}) \neq 0$, $p_m(X)$ is defined in Lemma~\ref{L01}   & $1$   &  \\
    & $n=2m, c \in  \F_{2^n} \setminus \F_{2^m}, \delta \in \F_{2^n} \setminus \F_{2^m}$ with $\Tr_{1}^{2m}(\delta)=\Trm(1)$ and $p_m(\delta+\bar\delta)^{-1}) \neq 0$, $p_m(X)$ is defined in Lemma~\ref{L01}  & $\leq 4$   &  \\
    \hline
    $(X^{2^m}+X+\delta)^{3*2^{2m-2}+2^{m-2}}+X$ 
    & $n=2m, m \not \equiv 0 \pmod 3, c \in  \F_{2^m} \setminus \{1\}, \delta \in \F_{2^n}$   & $1$   & Theorem~\ref{T3} \\
    \hline
    $(X^{2^m}+X+\delta)^{3*2^{m-2}+2^{2m-2}}+X$ 
    & $n=2m, m \not \equiv 0 \pmod 3, c \in  \F_{2^m} \setminus \{1\}, \delta \in \F_{2^n}$   & $1$   &  Theorem~\ref{T4} \\
    \hline
    $(X^{2^m}+X+\delta)^{2^{2m+1}+2^{m}}+(X^{2^m}+X+\delta)^{2^{m+1}+2^{2m}}+X$ 
    & $n=3m, c \in  \F_{2^m} \setminus \{1\}, \delta \in \F_{2^n}, \Tr_{m}^{3m}(\delta)=0$   & $1$   &  Theorem~\ref{T5} \\
    \hline
 \end{tabular}
 \end{adjustbox}
 \end{center}
 \end{table}
 
\section{Permutation polynomial over $\F_{3^n}$ with low $c$-differential uniformity}\label{S4}

 In this section, we consider the $c$-differential uniformity of polynomial $F(X)=(X^{3^m}-X+\delta)^{3^{m}+4}+(X^{3^m}-X+\delta)^{5}+X$, which is a permutation, as stated in Lemma~\ref{L04}, over $\F_{3^{n}}$, where $\delta \in \F_{3^n}$ and $n=2m$. 
\begin{thm} \label{T6}
  Let $F(X)=(X^{3^m}-X+\delta)^{3^{m}+4}+(X^{3^m}-X+\delta)^{5}+X$ over $\F_{3^{n}}$, where $n=2m$ and $\delta\in\F_{3^n}$ such that $(1-[\Trnn (\delta)]^4 ) $ is a square element in $\F_{3^m}^{*}$.  Then$:$
  \begin{enumerate}
   \item[$(1)$] $F$ is  P$c$N for all $c \in \F_{3^m} \setminus \{1\}$.
   \item[$(2)$] Moreover, the $c$-differential uniformity of $F$ is $3$, for all $c \in \F_{3^n}\setminus\F_{3^m}$.    
 \end{enumerate}
  \end{thm}
\begin{proof} 
Clearly, after simplifying
\allowdisplaybreaks\begin{align*}
 F(X)= & \Trnn(\delta) (X^{4\cdot 3^m}+X^4-X^{3^{m+1}+1}-X^{3^m+3} + \delta X^{3^{m+1}} + \delta^3 X^{3^m}+\delta^4)\\
  & \qquad \qquad  + \Trnn(\delta) (-\delta X^3 - \delta^3 X) + X.
\end{align*}

We know that for any $(a, b) \in \F_{3^{n}} \times \F_{3^{n}}$,
the $c$-DDT entry $_c\Delta_F(a, b)$ is given by the number of solutions $X \in \F_{3^{n}}$ of the following equation
\begin{equation}\label{eq31}
 F(X+a)-cF(X)=b,
\end{equation}
or equivalently,
\begin{equation*}
 (1-c)F(X)+ \Trnn(\delta)\Trnn ( (a^3-a^{3^{m+1}})X+a(X^3-X^{3^{m+1}})) + F(a)- \delta^{3^m+4}-\delta^5 = b.
 \end{equation*}
Now, by using Equation~\eqref{ddtw}, the number of solutions $X \in \F_{3^n}$, $_c\Delta_F(a,b)$, of the above Equation~\eqref{eq31} is given by 
\allowdisplaybreaks
\begin{align*}
 & \displaystyle{\frac{1}{3^n} \sum_{\beta \in \F_{3^n}} \sum_{X \in \F_{3^n}} \omega^{\displaystyle{\Tr\left(\beta((1-c)F(X)+ \Trnn(\delta)\Trnn ( (a^3-a^{3^{m+1}})X+a(X^3-X^{3^{m+1}})) )\right)}}} \\
 & \hspace{3cm} \omega^{\displaystyle{\Tr(\beta(F(a)- \delta^{3^m+4}-\delta^5-b))}},
\end{align*}
where $\omega =e^{2\pi i/3}$. Equivalently,
\allowdisplaybreaks
\begin{align*}
  _c\Delta_F(a,b) & = \dfrac{1}{3^n} \sum_{\beta \in \F_{3^n}} \omega^{\displaystyle{\Tr\left(\beta\left(F(a)-b- \delta^{3^m+4}-\delta^5 \right)\right)}}\\
& \sum_{X \in \F_{3^n}} \omega^{\displaystyle{\Tr\left(\beta((1-c)F(X)+ \Trnn(\delta)\Trnn ( (a^3-a^{3^{m+1}})X+a(X^3-X^{3^{m+1}})) )\right)}}.
 \end{align*}
Let 
\allowdisplaybreaks
\begin{align*}
 T_0 & = \Tr(\beta(1-c)F(X)), \\
 T_1 & = \Tr(\beta \Trnn(\delta)\Trnn ( (a^3-a^{3^{m+1}})X+a(X^3-X^{3^{m+1}}))).
\end{align*}
Then,
\begin{equation}
\label{eq35}
 _c\Delta_F(a,b) = \frac{1}{3^n} \sum_{\beta \in \F_{3^n}} \omega^{\displaystyle{\Tr\left(\beta\left(F(a)-b- \delta^{3^m+4}-\delta^5 \right)\right)}}\sum_{X \in \F_{3^n}} \omega^{\displaystyle{T_0+T_1}}.
\end{equation}

\textbf{Case 1.} Let $c\in \F_{3^m}\setminus \{1\}$ and $\delta \in \F_{3^n}$. To compute $T_0$ and $T_1$, we first write
\allowdisplaybreaks
\begin{equation*} 
\begin{split}
T_{1} & =  \Tr(\beta \Trnn(\delta)\Trnn ( (a^3-a^{3^{m+1}})X+a(X^3-X^{3^{m+1}}))) \\
 & = \Tr\left(\Trnn(\beta)\Trnn(\delta)(a^3-a^{3^{m+1}})X+\Trnn(\beta^{3^{m-1}})\Trnn(\delta^{3^{m-1}})(a^{3^m}-a)^{3^{m-1}}X\right),
\end{split}
\end{equation*}
and
\allowdisplaybreaks
\begin{align*} 
T_{0} & =  \Tr(\beta(1-c)F(X)) \\
 & = \Tr\left(\beta(1-c)\Trnn(\delta) (X^{4\cdot 3^m}+X^4-X^{3^{m+1}+1}-X^{3^m+3} + \delta X^{3^{m+1}} + \delta^3 X^{3^m}+\delta^4)\right.\\
 &\qquad\qquad \left. +  \Trnn(\delta) (-\delta X^3 - \delta^3 X) + X\right)\\
 & = \Tr\left(\beta(1-c)(\delta^{3^m+4}+\delta^5)\right)+ \Tr\left(\Trnn(\delta) \Trnn(\beta(1-c))X^4 +\beta(1-c)X \right.\\
 & \qquad \left.- \Trnn(\delta)^{3^{m-1}} \Trnn(\beta(1-c))^{3^{m-1}}X^{3^{m-1}+1} + (\Trnn(\delta) \delta \beta(1-c))^{3^{m-1}}X\right.\\
 & \left. + (-(\Trnn(\delta) \delta \beta(1-c))^{3^{2m-1}}+(\Trnn(\delta) \delta^3 \beta(1-c))^{3^m}-\Trnn(\delta) \delta^3 \beta(1-c))X \right).
\end{align*}

Now Equation~\eqref{eq35} reduces to
\allowdisplaybreaks
\begin{align*}
 _c\Delta_F(a,b) & =\frac{1}{3^n} \sum_{\beta \in \F_{3^n}} \omega^{\displaystyle{\Tr\left(\beta(F(a)-b- c(\delta^{3^m+4}+\delta^5))\right)}}\\
 & \qquad \qquad \qquad \sum_{X \in \F_{3^n}} \omega^{\displaystyle{\Tr(u_1 X^4-u_2 X^{3^{m-1}+1}+vX)}},
\end{align*}
where 
\allowdisplaybreaks
\begin{align*}
 u_1 & = \Trnn(\delta) \Trnn(\beta(1-c)) \\
 u_2 & = \Trnn(\delta)^{3^{m-1}} \Trnn(\beta(1-c))^{3^{m-1}} = u_1^{3^{m-1}}\\
 v & = \Trnn(\beta)\Trnn(\delta)(a^3-a^{3^{m+1}})+\Trnn(\beta^{3^{m-1}})\Trnn(\delta^{3^{m-1}})(a^{3^m}-a)^{3^{m-1}} +\beta(1-c)\\ 
 & + (\Trnn(\delta) \delta \beta(1-c))^{3^{m-1}} -(\Trnn(\delta) \delta \beta(1-c))^{3^{2m-1}}+(\Trnn(\delta) \delta^3 \beta(1-c))^{3^m}\\
 & \qquad \qquad -\Trnn(\delta) \delta^3 \beta(1-c).
\end{align*}

Further, splitting the above sum depending on whether $\Trnn(\beta)$ is~$0$ or not, we get
\allowdisplaybreaks
\begin{align*}
 S_0 & = \sum_{\substack{\beta \in \F_{3^n} \\ \Trnn(\beta)=0}} \omega^{\displaystyle{\Tr(\beta(F(a)-b- c(\delta^{3^m+4}+\delta^5)))}}\\
 & \sum_{X \in \F_{3^n}} \omega^{\displaystyle{\Tr\left(\left((1-c)\beta-(1-c)\beta\Trnn(\delta)^4 + (1-c)^{3^{m-1}}\beta^{3^{m-1}}\Trnn(\delta)^{2 \cdot 3^{m-1}} \right)X\right)}} \\
 & = 3^n + \sum_{\substack{{\beta \in \F_{3^n}^{*}}\\ \Trnn(\beta)=0}} \omega^{\displaystyle{\Tr(\beta(F(a)-b- c(\delta^{3^m+4}+\delta^5)))}}\\
 & \sum_{X \in \F_{3^n}} \omega^{\displaystyle{\Tr\left(\left((1-c)\beta-(1-c)\beta\Trnn(\delta)^4 + (1-c)^{3^{m-1}}\beta^{3^{m-1}}\Trnn(\delta)^{2 \cdot 3^{m-1}} \right)X\right)}}. 
\end{align*}
To compute $S_0$, we need to find the number of solutions $\beta \in \F_{3^n}$ for the following equation
$$
(1-c)\beta-(1-c)\beta\Trnn(\delta)^4 + (1-c)^{3^{m-1}}\beta^{3^{m-1}}\Trnn(\delta)^{2 \cdot 3^{m-1}} = 0.
$$
It is clear that when $\Trnn(\delta)=0$, the above equation has only one solution $\beta=0$. Let us assume that $\Trnn(\delta) \neq 0$. Raising the above equation to a cubic  power, we get
$$
(1-c)^3\beta^3-(1-c)^3\beta^3(\Trnn(\delta)^4)^3 - (1-c) \beta \Trnn(\delta)^{2} = 0,
$$
or equivalently, $A \beta^3 - B \beta =0$, where
$ A = (1-c)^3 \left(1-\Trnn(\delta)^4\right)^3$ and $B = (1-c) \Trnn(\delta)^2$. One can clearly observe that $A \beta^3 - B \beta =0$ has a solution in $\beta \in \F_{3^n}^{*}$ if and only if $\dfrac{B}{A}$ is a square in $\F_{3^m}$. Now $\dfrac{B}{A} = \dfrac{\Trnn(\delta)^2}{(1-c)^2(1-[\Trnn(\delta)^4])^3} = \left(\dfrac{\Trnn(\delta)}{(1-c)\nu^3}\right)^2$, where $\nu \in \F_{3^m}^{*}$ and $\nu^2 =  (1-\Trnn(\delta)^4)$. Hence $\beta = \pm \dfrac{\Trnn(\delta)}{(1-c)\nu^3}$, however, $\Trnn(\beta) \neq 0$ as $\Trnn(\beta) = \pm 2  \dfrac{\Trnn(\delta)}{(1-c)\nu^3}$. This would give us a contradiction to the condition that $\Trnn(\beta)=0$. Hence $S_0=3^n$.

Next we compute sum $S_1$,
\allowdisplaybreaks\begin{align*}
 S_1 & = \sum_{\substack{\beta \in \F_{3^n} \\ \Trnn(\beta)\neq 0}} \omega^{\displaystyle{\Tr(\beta(F(a)-b- c(\delta^{3^m+4}+\delta^5)))}} \sum_{X \in \F_{3^n}} \omega^{\displaystyle{\Tr\left(u_1 X^4-u_2 X^{3^{m-1}+1}+vX\right)}}\\
 & = \sum_{\substack{\beta \in \F_{3^n} \\ \Trnn(\beta)\neq0}} \omega^{\displaystyle{\Tr(\beta(F(a)-b- c(\delta^{3^m+4}+\delta^5)))}} \mathcal{W}_G(-v),
 \end{align*}
 where $\mathcal{W}_G(-v)$ is the Walsh transform of trace of function $G ( X) = u_1 X^4-u_2 X^{3^{m-1}+1}$ at $-v$. Now, from Lemma ~\ref{owalsh}, the absolute square of the Walsh transform coefficient of $G$ is given by
 \begin{equation*} \lvert \mathcal{W}_G(-v) \rvert^2 =
  \begin{cases}
   3^{n+\ell} &~\mbox{if}~G(X)+\Tr(vX)\equiv0~\mbox{on Ker}~(L),  \\
    0 &~\mbox{otherwise},
  \end{cases}
 \end{equation*}
where $\ell$ is dimension of kernel of the linearized polynomial 
 $$
 L(X)=u_1(X-X^{3^m})^3-u_2(X-X^{3^m})^{3^{m-1}}.
 $$ 
It is obvious that $\F_{3^m} \subseteq$ Ker$(L) $. Therefore, if we show that $G(X)+\Tr(vX) \neq 0$ for all $X \in \F_{3^m}$, then $S_1=0$. We shall now prove that $G(X)+\Tr(vX)$ is not identically zero on $\F_{3^m}$. For $X \in \F_{3^m}$, the polynomial $G(X)+\Tr(vX)$ gets reduced to the polynomial 
\[
u_1 X^4-u_2 X^{3^{m-1}+1}+ \Trnn(\beta(1-c))X+ \Trnn(\beta(1-c))^3 X^3+\cdots+\Trnn(\beta(1-c))^{3^{m-1}}X^{3^{m-1}},
\]
 which is a polynomial over $\F_{3^m}$. For $m=1$, $G(X)+\Tr(vX)$ is not identically zero on $\F_{3}$ as $\Trnn(\beta(1-c)) \neq 0$. And for $m \geq 2$, it is easy to observe that degree of $G(X)+\Tr(vX)$ is strictly less than $3^{m}-1$ and hence the claim is shown.

\textbf{Case 2.} Let $c \in \F_{3^n}\setminus\F_{3^m}$. Then using the same values of $u_1, u_2$ and $v$ as above, one can define 
the sums $S_0$ and $S_1$ depending upon $\Trnn(\beta(1-c))=0$ or $\Trnn(\beta(1-c)) \neq 0$. First,
\allowdisplaybreaks
\begin{align*}
 S_0 & = \sum_{\substack{\beta \in \F_{3^n} \\ \Trnn(\beta(1-c))=0}} \omega^{\displaystyle{\Tr(\beta(F(a)-b- c(\delta^{3^m+4}+\delta^5)))}} \sum_{X \in \F_{3^n}} \omega^{\displaystyle{\Tr\left(vX\right)}}\\
 & = 3^n + \sum_{\substack{\beta \in \F_{3^n}^{*}\\ \Trnn(\beta(1-c))=0}} \omega^{\displaystyle{\Tr(\beta(F(a)-b- c(\delta^{3^m+4}+\delta^5)))}} \sum_{X \in \F_{3^n}} \omega^{\displaystyle{\Tr\left(vX\right)}}.
\end{align*}
To compute $S_0$, we need to find the number of solutions for
\allowdisplaybreaks\begin{align*}
 & \Trnn(\beta)\Trnn(\delta)(a^3-a^{3^{m+1}})+\Trnn(\beta^{3^{m-1}})\Trnn(\delta^{3^{m-1}})(a^{3^m}-a)^{3^{m-1}} +\beta(1-c)\\ 
 & + (\Trnn(\delta) \delta \beta(1-c))^{3^{m-1}} -(\Trnn(\delta) \delta \beta(1-c))^{3^{2m-1}}+(\Trnn(\delta) \delta^3 \beta(1-c))^{3^m}\\
 & \qquad \qquad -\Trnn(\delta) \delta^3 \beta(1-c) =0,
\end{align*}
or equivalently, using $\Trnn(\beta(1-c))=0$, we have
\allowdisplaybreaks\begin{align*}
 & (1-\tilde c) \Trnn(\delta)(a-a^{3^m})^3 \beta + (1-\tilde c)^{3^{m-1}}\Trnn(\delta^{3^{m-1}})(a^{3^m}-a)^{3^{m-1}} \beta^{3^{m-1}} +\beta(1-c)\\ 
 & \qquad \qquad  + (\Trnn(\delta)^{2 \cdot 3^{m-1}} (1-c)^{3^{m-1}} \beta^{3^{m-1}} -\Trnn(\delta)^4 (1-c) \beta =0,
\end{align*}
where $\tilde c = (1-c)^{1-3^m}$. Raising the above equation  to the cubic power, we get $A\beta^3 - B\beta=0$, where 
$$
A= (1-\tilde c)^3 \Trnn(\delta^3) (a-a^{3^m})^9 +(1-c)^3 - (1-c)^3 (\Trnn(\delta))^{12}
$$
and
$$
B = \tilde c (1 - \tilde c)^{3^m} \Trnn(\delta) (a-a^{3^m}) + \tilde c (1 -  c)^{3^m} \Trnn(\delta)^2.
$$
It is easy to see that except for $\beta=0$, $A\beta^3 -B\beta=0$ has a solution $\beta \in \F_{3^n}^{*}$ if $\dfrac{B}{A}$ is a square (notice that for $a \in \F_{3^m}$,$\dfrac{B}{A}$ is always a square). If it is a square, then $A\beta +B\beta^3=0$ has three solution in $\F_{3^n}$, namely, $\beta=0, \beta =\beta_1$ and $\beta = -\beta_1$. Hence, $S_0$ becomes  
\allowdisplaybreaks
\begin{align*}
& 3^n\left(1 + \omega^{\displaystyle{\Tr(\beta_1(F(a)-b- c(\delta^{3^m+4}+\delta^5)))}} +  \omega^{\displaystyle{\Tr(-\beta_1(F(a)-b- c(\delta^{3^m+4}+\delta^5)))}}\right).
\end{align*}
Clearly, for those pairs $(a,b)\in\F_{3^n}\times\F_{3^n}$ for which $b=F(a)- c(\delta^{3^m+4}+\delta^5)$, we have $S_0=3^{n+1}$; and for the other pairs  $(a,b) \in \F_{3^n} \times \F_{3^n}$, we have $S_0=3^n(1+\omega^{\Tr(\alpha)}+\omega^{\Tr(-\alpha)})$, where $\alpha = \beta_1(F(a)-b- c(\delta^{3^m+4}+\delta^5))$. Hence, the maximum value that $S_0=3^n(1+\omega^{\Tr(\alpha)}+\omega^{\Tr(-\alpha)})$ can attain is $3^{n+1}$, as $\Tr(-\alpha)=-\Tr(\alpha)$. This yields that $S_0=3^{n+1}$.

Next, we analyze the second sum,
\allowdisplaybreaks
 \begin{align*}
 S_1 & = \sum_{\substack{\beta \in \F_{3^n} \\ \Trnn(\beta(1-c))\neq0}} \omega^{\displaystyle{\Tr(\beta(F(a)-b- c(\delta^{3^m+4}+\delta^5)))}} \\
 & \qquad \sum_{X \in \F_{3^n}} \omega^{\displaystyle{\Tr(u_1 X^3-u_2 X^{3^{m-1}+1}+vX)}}\\
 & = \sum_{\substack{\beta \in \F_{3^n} \\ \Trnn(\beta(1-c))\neq0}} \omega^{\displaystyle{\Tr(\beta(F(a)-b- c(\delta^{3^m+4}+\delta^5)))}} \mathcal{W}_G(-v).
 \end{align*} 
where $\mathcal{W}_G(-v)$ is the Walsh transform of trace of function $G : X \mapsto u_1 X^4-u_2 X^{3^{m-1}+1}$ at $-v$. By following  similar arguments as in the Case 1 above, one can show that $S_1=0$. This completes the proof. 
\end{proof}

Here, we give an example for the above Theorem~\ref{T6}
\begin{exmp}
Let $F(X)=(X^{9}-X+g^{10})^{13}+(X^{9}-X+g^{33})^{5}+X$, a permutation polynomial over $\F_{3^{4}}$ and $\F_{3^{4}}^{*}=\langle g \rangle $, where $g$ is a primitive element of $\F_{3^{4}}$. Observe that $1-(\Tr_{2}^{4}(g^{10}))^4$ is square of $2g^3 + 2g^2 + 2 \in \F_{3^2}^{*}$. Then, $F(X)$ is P$c$N for all $c \in \{0,
g^{10}, g^{20}, g^{30}, g^{40}, g^{50}, g^{60}, g^{70}, g^{80}\}$ and has $c$-differential uniformity $3$ for all $c \in \F_{3^4}\setminus \{0,1, g^{10}, g^{20}, g^{30}, g^{40}, g^{50}, g^{60}, g^{70}, g^{80}\}$.
\end{exmp}

\section{Conclusions}
\label{sec:concl}
In this paper we show that some permutation polynomials are P$c$N over finite fields of even characteristic and even dimension, for $c\neq 1$ in the subfield of half dimension. This adds to the small list of known (non-trivial) P$c$N functions. We also find a class of permutation polynomials over finite fields of characteristic~$3$, of even dimension $n=2m$, which is P$c$N for $c\in\F_{3^m}\setminus \{1\}$, and has $c$-differential uniformity~$3$ for all $c\notin\F_{3^m}$. We use Walsh transforms computations, Weil sums and other number theoretical techniques to deal with some very delicate equations, which could have an interest of its own.

\section*{Acknowledgments}
We sincerely thank the editors for handling our paper, and the referees for their careful reading, valuable comments, and constructive suggestions.

\end{document}